\documentclass[12pt]{article}
\usepackage{amsmath}
\usepackage{graphicx,psfrag,epsf}
\usepackage{enumerate}
\usepackage{natbib}
\usepackage{url} 

\usepackage{mathrsfs,amsfonts,amssymb,amsthm}
\usepackage{epsfig}

\usepackage{mdwmath}
\usepackage{mdwtab}
\usepackage{subfigure}
\usepackage{color,soul}
\usepackage{amsmath,amsthm, amssymb,epsfig,epsf,psfrag, color,soul,cancel}
\usepackage{url}            
\usepackage{booktabs}       
\usepackage{amsfonts}       
\usepackage{nicefrac}       
\usepackage{microtype}      
\usepackage{graphicx, epsfig}
\usepackage{amsthm}
\usepackage{amsmath}

\newtheorem{assumption}{Assumption}
\newtheorem{theorem}{Theorem}
\newtheorem{corollary}{Corollary}
\newtheorem{lemma}{Lemma}

\newcommand{\vv}[1]{\mathbf{#1}}

\newcommand{\blind}{1}

\begin{document}

\def\spacingset#1{\renewcommand{\baselinestretch}%
{#1}\small\normalsize} \spacingset{1}


\if1\blind
{
  \title{\bf Robustness and Tractability for Non-convex M-estimators}
  \author{Ruizhi Zhang,
    Yajun Mei,
    Jianjun Shi,
    and
    Huan Xu\\
    H. Milton Stewart School of Industrial and Systems Engineering, \\
    Georgia Institute of Technology
    }
  \maketitle
} \fi

\if0\blind
{
  \bigskip
  \bigskip
  \bigskip
  \begin{center}
    {\LARGE\bf Robustness and Tractability for Non-convex M-estimators}
\end{center}
  \medskip
} \fi

\bigskip
\begin{abstract}
We investigate two important properties of M-estimator, namely, robustness and tractability, in  linear regression setting, when the observations are contaminated by some arbitrary outliers.
Specifically, robustness means the statistical property that the estimator should always be close to the underlying true parameters {\em regardless of the distribution of the outliers}, and tractability indicates the computational property that the estimator can be computed efficiently, even if the objective function of the M-estimator is {\em non-convex}.
In this article, by learning the landscape of the empirical risk, we show that under mild conditions, many M-estimators enjoy nice robustness and tractability properties simultaneously, when the percentage of outliers is small.
We further extend our analysis to the high-dimensional setting, where the number of parameters is greater than the number of samples, $p \gg n$, and prove that when the proportion of outliers is small, the penalized M-estimators with {\em $L_1$} penalty will enjoy robustness and tractability simultaneously. Our research provides an analytic approach to see the effects of outliers and tuning parameters on the robustness and tractability  for some families of  M-estimators. Simulation and case study are presented to illustrate the usefulness of our theoretical results for M-estimators under Welsch's exponential squared loss.
\end{abstract}

\noindent%
\vfill

\newpage
\spacingset{1.45} 
\section{Introduction}


M-estimation plays an important role in linear regression due to its robustness and flexibility.
From the statistical viewpoint, it has been shown that many M-estimators enjoy desirable robustness properties in the presence of outliers, as well as asymptotic normality when the data are normally distributed without outliers. Some general theoretical properties and review of robust M-estimators can be found in \cite{bai:1992,huber:2009,cheng:2010,hampel:2011,el:2013}. In the high-dimensional setting, where the dimensionality is greater than the number of samples, penalized M-estimators have been widely used to tackle the challenges of outliers and have been used for sparse recovery and variable selection, see \cite{lambert:2011,li:2011,wang:2013,loh:2017}.   However, from the computational tractability perspective, it is often not easy to compute the M-estimators, since optimization problems over non-convex loss functions are usually involved. Moreover, the tractability issue may become more challenging when the data are contaminated by some arbitrary outliers, which is essentially the situation where robust M-estimator is designed to tackle.

This paper aims to investigate two important properties of   M-estimators, {\em robustness} and {\em tractability}, simultaneously under {\em the gross error model}. Specifically, we assume the data generation model is $y_i=\langle\theta_0, x_i\rangle+\epsilon_i,$ where $y_i\in \mathbb{R}, x_i \in \mathbb{R}^p,$ , for $i = 1, \cdots, n$,  and the noise term $\epsilon_i$'s are from  Huber's gross error model \citep{huber:1964}:  $\epsilon_i\sim (1-\delta)f_0 +\delta g$, for $i = 1, \cdots, n$. Here, $f_0$  denotes the probability density function (pdf) of the noise of the normal samples, which has  the desirable properties, such as zero mean and finite variance; $g$ denotes the pdf  of the outliers (contaminations), which may also depend on the explanatory variable $x_i$, for $i = 1, \cdots, n$. One thing to notice is that we do not require the mean of $g$ to be $0.$  The parameter  $\delta \in [0,1]$, denotes the percentage of the contaminations, which is also known as the contamination ratio in robust statistics literature.   The gross error model indicates that for the $i^{th}$ sample, the residual term $\epsilon_i$ is generated from the pdf $f_0$ with probability $1-\delta,$  and from the pdf $g$ with probability $\delta.$ It is important to point out that the residual $\epsilon_i$ is independent of $x_i$ and other $x_{j}$'s when it is from the pdf $f_0,$
but can be dependent with the variable $x_i$ when it is from the pdf $g.$

In the first part of this paper, we start with the low-dimensional case when the dimension $p$ is fixed. We consider the robust M-estimation with a constraint on the $\ell_2$ norm of $\theta$. Mathematically, we study the following optimization problem:
\begin{eqnarray}\label{form}
&&\underset{\theta}{\mbox{Minimize:}}\quad\hat{R}_{n}(\theta):=\frac{1}{n} \sum_{i=1}^{n}\rho(y_i-\langle\theta, x_i\rangle),\\
&&\text{subject to:     }\quad \|\theta\|_2\le r.\nonumber
\end{eqnarray}
Here, $\rho: \mathbb{R}\to \mathbb{R}$ is the loss function, and  is often {\em non-convex}.
We consider the problem with the  $\ell_2$ constraint due to three reasons: first, it is well know the constrainted optimization problem in (\ref{form}) is equivalent to the unconstrained optimization problem with a $\ell_2$ regularizer. Therefore, it is related to the Ridge regression, which can alleviate multicollinearity amongst regression predictors. Second, by considering the problem of (\ref{form}) in a compact ball with radius $r,$ it guarantees the existence of the global optimal, which is necessary for establishing the tractability properties of the M-estimator. Finally, by working on the constrained optimization problem, we can avoid technical complications and establish the uniform convergence theorems of the empirical risk and population risk. Besides, the constrained M-estimators are widely used and studied in the literature, see \citet{geyer:1994,mei:2016,loh:2017} for more details. To be consistent with the assumptions used in the literature, in the current work, we assume $r$ is a constant and the true parameter $\theta_0$ is inside of the ball.

In the second part, we extend our research to the high-dimensional case, where $p\gg n$ and the true parameter  $\theta_0$ is sparse.
In order to achieve the sparsity in the resulting estimator, we consider the penalized M-estimator with the $\ell_1$ regularizer:
\begin{eqnarray}\label{form33}
&&\underset{\theta}{\mbox{Minimize:}}\quad\hat{L}_{n}(\theta):=\frac{1}{n} \sum_{i=1}^{n}\rho(y_i-\langle\theta, x_i\rangle)+\lambda_n||\theta||_1,\\
&&\text{subject to:     }\quad \|\theta\|_2\le r.\nonumber
\end{eqnarray}
Note the corresponding penalized M-estimator with the $\ell_2$ constraint is related to the Elastic net, which overcomes the limitations of the LASSO type regularization \citep{zou:2005}.

In both parts, we will show that (in the finite sample setting,) the M-estimator obtained from (\ref{form}) or (\ref{form33}) is robust in the sense that all stationary points of empirical risk function $\hat{R}_n(\theta)$ or $\hat{L}_n(\theta)$ are bounded in the neighborhood of the true parameter $\theta_0$  when the proportion of outliers is small.  In addition, we will show that with a high probability,  there is a unique stationary point of the empirical risk function, which is the global minimizer of (\ref{form}) or (\ref{form33}) for some general (possibly nonconvex) loss functions $\rho$. This implies that the M-estimator can  be computed efficiently. To illustrate our general theoretical results, we study some specific M-estimators with Huber's loss \citep{huber:1964} and Welsch's exponential squared loss \citep{dennis:1978}, and explicitly discuss how the tuning parameter and percentage of outliers affect the robustness and tractability of the corresponding M-estimators.

Our research makes several fundamental contributions on the field of robust statistics and non-convex optimization.
First, we demonstrate the uniform convergence results for the gradient and Hessian of the empirical risk to the population risk under the gross error model. Second, we provide nonasymptotic upper bound of the estimation error for the general M-estimators, which nearly achieve the minimax error bound in \cite{chen:2016}. Third, we investigate the computational tractability of the general non-convex M-estimators under the gross error model and show when the contamination ratio $\delta$ is small, there is only one unique stationary point of the empirical risk function. Therefore, efficient algorithms such as gradient descent or proximal gradient decent can be guaranteed to  converge to a
unique global minimizer irrespective of the initialization. Our general results also imply the following interesting and to some extent surprising statement: the percentage of outliers has an impact on the {\em tractability} of non-convex M-estimators. In a nutshell, the estimation and the corresponding optimization problem become more difficult both in terms of solution quality and computational efficiency when more outliers appear. While the former is well expected, we find the latter~--~that more outliers make M-estimators more difficult to numerically compute --~an interesting and somewhat surprising  discovery.  Our simulation results and case study also verify this phenomenon.

\subsubsection*{Related works}
Since Huber's pioneer work on robust M-estimators \citep{huber:1964}, many M-estimators with different choices of loss functions have been proposed, e.g.,  Huber's loss \citep{huber:1964}, Andrew's sine loss \citep{andrews:1972}, Tukey's Bisquare loss \citep{beaton:1974}, Welsch's exponential squared loss \citep{dennis:1978}, to name a few.  From the statistical perspective, much research has been done to investigate the robustness of M-estimators such as large breakdown point \citep{donoho:1983,mizera:1999,alfons:2013}, finite influent function \citep{hampel:2011} and asymptotic normality \citep{maronna:1981,lehmann:2006,el:2013}. Recently, in the high-dimensional context, regularized M-estimators have received a lot of attentions. \cite{lambert:2011} proposed a robust variable selection method by combing Huber's loss and adaptive lasso penalty.  \cite{li:2011} show the nonconcave penalized M-estimation method can perform parameter estimation and variable selection simultaneously.
Welsch's exponential squared loss combined with adaptive lasso penalty is used by \cite{wang:2013} to construct a robust estimator for sparse estimation and variable selection. \cite{chang:2018} proposed a robust estimator by combining the Tukey's biweight loss with adaptive lasso penalty. \cite{loh:2015} proved that under mild conditions,  any stationary point of the non-convex objective function will close to the underlying true parameters. However, those statistical works did not discuss the computational tractability of the M-estimators even though many of these loss functions are non-convex.

During the last several years, non-convex optimization has attracted fast growing interests due to its ubiquitous applications in machine learning and in particular deep learning, such as dictionary learning \citep{mairal:2009}, phase retrieval \citep{candes:2015}, orthogonal tensor decomposition \citep{anandkumar:2014} and training deep neural networks \citep{bengio:2009}.
It is well known that there is no efficient algorithm  that can guarantee to find the global optimal solution for  general non-convex optimization.

Fortunately, in the context of estimating non-convex M-estimators for high-dimensional linear regression ({\em without outliers}), under some mild statistical assumptions, \cite{loh:2017} establishes the uniqueness of the stationary point of the non-convex M-estimator when using some non-convex bounded regularizers instead of $\ell_1$ regularizer. By investigating the uniform convergence of gradient and Hessian of the empirical risk, \citet*{mei:2016} prove that with a high probability, there exists one unique stationary point of the regularized empirical risk function with $\ell_1$ regularizer. Thus regardless of the initial points, many computational efficient algorithm such as gradient descent or proximal gradient descent algorithm could be applied and are guaranteed to converge to the global optimizer, which implies the high tractability of the M-estimator. However, their analysis is restricted to the standard linear regression setting without outliers. In particular, they assume the distribution of the noise terms in the linear regression model should have some desirable properties such as zero mean, sub-gaussian and independent of feature vector $x$, which might not hold when the data are contaminated with outliers.  To the best of our knowledge, no research has been done on analyzing the computational tractability properties of the non-convex M-estimators when data are contaminated by arbitrary outliers, although the very reason why M-estimators are proposed is to handle outliers in linear regression in the robust statistics literature. Our research is the first to fill the significant gap on the tractability of non-convex M-estimators. We prove that under mild assumptions, many M-estimators can tolerate a small amount of arbitrary outliers in the sense of keeping the tractability, even if the loss functions are non-convex.

\textbf{Notations.}
 Given $\mu,\nu\in \mathbb{R}^p,$ their standard inner product is defined by $\langle \mu,\nu \rangle=\sum_{i=1}^{p}\mu_i\nu_i.$ The $\ell_p$ norm of a vector $x$ is denoted by $||x||_p.$ The $p$ by $p$ identity matrix is denoted by $I_{p\times p}.$ Given a matrix $M\in \mathbb{R}^{m\times m},$ let $\lambda_{\max}(M), \lambda_{\min}(M)$ denote the largest and the smallest eigenvalue of $M$,  respectively. The operator norm of $M$ is denoted by $||M||_{op},$ which is equal to $\max(\lambda_{\max}(M),-\lambda_{\min}(M))$ when $M\in \mathbb{R}^{m\times m}.$  Let $B_q^p(a,r)=\{x\in \mathbb{R}^p: ||x-a||_q\le r\}$, be the $\ell_q$ ball in the $\mathbb{R}^p$ space with  center $a$ and radius $r.$ Given a random variable $X$ with probability density function $f,$ we denote the corresponding expectation by $\vv E_f.$ We will often omit the density function subscript $f$ when it is clear from the context, the expectation is taken for all variables.

\textbf{Organization.} The rest of this article is organized as follows. In Section 2, we present the theorems about the robustness and tractability of general M-estimators under the low-dimensional setup when dimension $p$ is fixed and less than $n.$ Then in Section 3, we consider the penalized M-estimator with $\ell_1$ regularizer in the high-dimensional regression when $p\gg n.$ The $\ell_2$ error bounds of the estimation and the scenario when the M-estimator has nice tractability are provided. In Section 4, we discuss two special families of robust estimator constructed by Huber's and Welsch's exponential loss as examples to illustrate our general theorems of robustness and tractability of M-estimators. Simulation results are presented in Section 5 and a case study is shown in Section 6 to illustrate the robustness and tractability properties when the data are contaminated by outliers.  Concluding remarks are given in Section 7. We relegate all proofs to the Appendix due to space limits.

\section{M-estimators in the low-dimensional regime}\label{sec.M-est}
In this section, we investigate two key properties of M-estimators, namely {\em robustness} and {\em tractability}, in the setting of linear regression with arbitrary outliers in the low-dimensional regime where the dimension $p$ is fixed and smaller than the number of samples $n$. In terms of robustness, we  show that under some mild conditions, any stationary point of the objective function in (\ref{form}) will be well bounded in a neighborhood of the true parameter $\theta_0.$ Moreover, the  neighborhood shrinks when the proportion of outliers decreases.  In terms of tractability, we  show that when the proportion of outliers is small and the sample size is large, with  a high probability, there is a {\em unique stationary point} of the empirical risk function, which is the global optimum (and hence the corresponding M-estimator). Consequently, many first order methods are guaranteed to converge to the global optimum, irrespective of initialization.

Before presenting our main theorems, we make the following mild
assumptions on the loss function $\rho,$ the explanatory or feature vectors $x_i$, and the idealized noise distribution $f_0.$   We define the score function $\psi(z):=\rho'(z).$
\begin{assumption}\label{assume1}
  \begin{description}
  \item[(a)] The score function $\psi(z)$ is twice differentiable and odd in $z$ with $\psi(z)\ge 0$ for all $z\ge 0.$ Moreover, we assume
  $\max\{||\psi(z)||_{\infty},||\psi'(z)||_{\infty},||\psi''(z)||_{\infty}\}\le L_{\psi}.$
  \item[(b)] The feature vector $x_i$ are i.i.d with zero mean and $\tau^2$-sub-Gaussain, that is $\vv{E}[e^{\langle \lambda, x_i\rangle}]\le \exp(\frac{1}{2} \tau^2 ||\lambda||_2^2)$, for all $\lambda\in \mathbb{R}^p.$
  \item[(c)] The feature vector $x_i$ spans all possible directions in $\mathbb{R}^{p},$ that is $\vv E[x_i x_i^T]\succeq \gamma \tau^2 I_{p\times p}$, for some $0<\gamma\le1.$
  \item[(d)] The idealized noise distribution $f_0(\epsilon)$ is symmetric. Define
  $h(z):= \int_{-\infty}^{\infty}f_0(\epsilon)\psi(z+\epsilon)d \epsilon$ and $h(z)$ satisfies $h(z)>0$, for all $z>0$ and $h'(0)>0.$
\end{description}
\end{assumption}
Assumption (a) requires the smoothness of the loss function in the objective function, which is crucial to study the tractability of the estimation problem; Assumption (b) assumes the sub-Gaussian design of the observed feature matrix; Assumption (c) assumes that the covariance matrix of the feature vector is positive semidefinite.
We remark that the condition on $h(z)$ is mild. It is not difficult to show that it is satisfied if the idealized noise distribution $f_0(\epsilon)$ is strictly positive for all $\epsilon$ and decreasing for $\epsilon>0,$ e.g., if $f_{0}=$ pdf of $N(0, \sigma^2).$

Before presenting our main results in this section, we first define the population risk as follows:
\begin{eqnarray} \label{eqn04}
  R(\theta) = \vv E \hat{R}_{n}(\theta) =  \vv{E}[\rho(Y-\langle\theta, X\rangle)].
\end{eqnarray}

The high level idea is to analyze the population risk first, and then we build a link between the population risk and the empirical risk, which solves the original estimation problem. Theorem \ref{thm1} below summarizes the results for the population risk function $R(\theta)$ in (\ref{eqn04}).
\begin{theorem}\label{thm1}
Assume that Assumption \ref{assume1} holds and  the true parameter $\theta_0$ satisfies $||\theta_0||_2\le r/3.$
\begin{description}
\item[(a)] There exists a constant $\eta_0=\frac{\delta}{1-\delta} C_1$  such that any stationary point $\theta^*$ of $R(\theta)$ satisfies $||\theta^*-\theta_0||_2\le \eta_0,$ where $\delta$ is the contamination ratio, and $C_1$ is a positive constant that only depends on $\gamma, r, \tau,\psi(z)$ and the pdf $f_0$, but does not depend on the outlier pdf $g$.
\item[(b)]  When $\delta$ is small, there exist a constant $\eta_1=C_2-C_3\delta>0,$ where $C_2, C_3$ are two positive constants that only depend on $\gamma, r, \tau,\psi(z)$ and the pdf $f_0$ but not depend on the outlier pdf $g,$ such that
\begin{eqnarray}
\lambda_{\min}(\nabla^2 R(\theta)) >0
\end{eqnarray}
for every $\theta$ with $||\theta_0-\theta||_2<\eta_1$.
\item[(c)] There is a unique stationary point of $R(\theta)$ in the ball $B_2^p(0,r)$ as long as $\eta_0<\eta_1$ for a given contamination ratio $\delta.$
 \end{description}
\end{theorem}

It is useful to add some remarks for better understanding Theorem \ref{thm1}. First, recall that the noise term   $\epsilon_i$  follows the gross error model: $\epsilon_i\sim (1-\delta)f_0+\delta g,$ where the outlier pdf $g$ may also depend on $x_i.$ While the true parameter $\theta_0$ may no longer be the stationary point of the population risk function $R(\theta),$  Theorem \ref{thm1} implies that the stationary points of $R(\theta)$ will always bounded in a neighborhood of the true parameter $\theta_0$ when the percentage of contamination $\delta$ is small. This indicates the robustness of M-estimators in the population case.

Second, Theorem \ref{thm1} asserts that when there are no outliers, i.e., $\delta=0,$ the stationary point is indeed the true parameter $\theta_0.$ In addition, since the constant $\eta_0$ in (a) is an increasing function of $\delta$ whereas the constant $\eta_1$ in (b) is a decreasing function of $\delta,$ stationary points of $R(\theta)$ may disperse from the true parameter $\theta_0$ and the strongly convex region around $\theta_0$ will be decreasing, as the contamination ratio $\delta$ is increasing. This indicates the difficulty of optimization for large contamination ratio cases.

Third, part (c) is a direct result from part (a) and (b).  Note that $\eta_0(\delta=0)=0<\eta_1(\delta=0)=C_2,$ thus there exists a positive $\delta^*,$ such that $\eta_0<\eta_1$ for any $\delta<\delta^*.$ A simple lower bound on $\delta^{*}$ is $C_3 /(C_1+ C_2+C_3),$ since $C_1 \delta < (1- \delta)(C_2 - C_3 \delta)$ whenever $0 \le \delta \le C_3 /(C_1+ C_2+C_3).$

Our next step is to link the empirical risk function (and the corresponding M-estimator) with the population version. To this end, we need the following lemma, which shows the global uniform convergence theorem of the sample gradient and Hessian.
\begin{lemma}\label{lemma1}
Under Assumption \ref{assume1}, for any $\pi>0,$ there exists a constant $C_{\pi}$ depending on $\pi,\gamma,r,\tau, \psi(z), h(z)$ but independent of $p,n, \delta$ and $g,$ such that for any $\delta\ge 0,$ the following hold:
\begin{description}
  \item[(a)] The sample gradient converges uniformly to the population gradient in Euclidean norm, i.e., if $n\ge C_{\pi}p\log n,$ we have
  \begin{eqnarray}
\vv P\left(\underset{\theta\in B_2^{p}(0,r)}{\sup} ||\nabla \hat{R}_{n}(\theta) -\nabla R(\theta)||_2\le \tau \sqrt{\frac{C_{\pi}p\log n}{n}}\right)\ge 1-\pi.
  \end{eqnarray}
  \item[(b)] The sample Hessian converges uniformly to the population Hessian in operator norm, i.e., if $n\ge C_{\pi}p\log n,$ we have
  \begin{eqnarray}
\vv P\left(\underset{\theta\in B_2^{p}(0,r)}{\sup} ||\nabla^2 \hat{R}_{n}(\theta) -\nabla^2 R(\theta)||_{op}\le \tau^2 \sqrt{\frac{C_{\pi}p\log n}{n}}\right)\ge 1-\pi.
  \end{eqnarray}
\end{description}
\end{lemma}

%

We are now ready to present our main result about M-estimators by investigating the empirical risk function $\hat{R}_n(\theta).$
\begin{theorem}\label{thm2}
Assume  Assumption \ref{assume1} holds and $||\theta_0||_2\le r/3.$ Let us use the same notation $\eta_0$ and $\eta_1$ as in Theorem \ref{thm1}. Then for any $\pi>0$, there exist constant $C_\pi$ depends on $\pi,\gamma,r,\tau,\psi,f_0$ but independent of $n,p,\delta$ and $g$ ,  such that as $n\ge C_{\pi} p \log n,$  the following statements hold with probability at least $1-\pi:$
\begin{description}
  \item[(a)] for all $||\theta-\theta_0||_2>2\eta_0,$
  \begin{eqnarray}
  \langle \theta-\theta_0, \nabla \widehat{R}_n(\theta)\rangle> 0.
  \end{eqnarray}
  \item[(b)] for all $||\theta-\theta_0||_2\le \eta_1,$
  \begin{eqnarray}
  \lambda_{\min}(\nabla^2 \widehat{R}_n(\theta))>0.
  \end{eqnarray}
  \end{description}
   Thus, as long as $2\eta_0<\eta_1,$  $\widehat{R}_n(\theta)$ has a unique stationary point, which lies in the ball $B^{p}(0,r).$ This is the unique global optimal solution of (\ref{form}), and denote this unique stationary point by $\widehat{\theta}_n$.
  \begin{description}
  \item[(c)] There exists a positive constant $\kappa$ that depends on $\pi, \gamma, r,\psi,\delta,f_0$ but independent of $n,p$ and $g$, such that
  \begin{eqnarray}\label{eq11}
||\widehat{\theta}_n-\theta_0||_2\le \eta_0+\frac{4\tau}{\kappa}\sqrt{\frac{C_{\pi}p\log n}{n}}.
\end{eqnarray}
\end{description}
\end{theorem}

A few remarks are in order. First, since $\eta_0$ is independent of $n,p$ and $g,$  Theorem \ref{thm2}(a) asserts that the M-estimator which minimizes $\widehat{R}_n(\theta)$ is always bounded in the ball $B_2^p(\theta_0,2\eta_0),$ regardless of $g$ (and hence the outliers observed). This indicates the robustness of the M-estimator, i.e., the estimates are not severely skewed by a small amount of ``bad'' outliers.
Next, when the contamination ratio $\delta$ is small such that $2\eta_0<\eta_1,$ there is a unique stationary point of $\widehat{R}_n(\theta).$ Therefore, although the original optimization problem (\ref{form}) is non-convex and the sample contains some arbitrary outliers, the optimal solution of $\widehat{R}_n(\theta)$ can be computed efficiently via most off-the-shelf first-order algorithms such as gradient descent or stochastic gradient descent. This indicates the tractability of the M-estimator. Interestingly,  as in the population risk case, the tractability is closely related to the amount of outliers~--~the problem is easier to optimize when the data contains fewer outliers. Finally, when the number of samples $n\gg p\log n,$  the estimation error bound $\eta_0$ is as the order of $O(\delta+\sqrt{\frac{p\log n}{n}}),$ which nearly achieves the minimax lower bound of $O(\delta+\sqrt{\frac{p}{n}})$ in \cite{chen:2016}.

\section{Penalized M-estimator in the high-dimensional regime}\label{third.M-est}
In this section, we investigate the tractability and the robustness of the penalized M-estimator in the high-dimension region where the dimension of parameter $p$ is much greater than the number of samples $n.$ Specifically, we consider the same data generation model $y_i=\langle\theta_0, x_i\rangle+\epsilon_i,$ where $y_i\in \mathbb{R}, x_i \in \mathbb{R}^p,$  and the noise term $\epsilon_i$ are from  Huber's gross error model \citep{huber:1964}:  $\epsilon_i\sim (1-\delta)f_0 +\delta g.$ Moreover, we assume $p\gg n$ and the true parameter $\theta_0$ is sparse.

We consider the $\ell_1$-regularized M-estimation under a $\ell_2$-constraint on $\theta$:
\begin{eqnarray}\label{form3}
&&\underset{\theta}{\mbox{Minimize:}}\quad\hat{L}_{n}(\theta):=\frac{1}{n} \sum_{i=1}^{n}\rho(y_i-\langle\theta, x_i\rangle)+\lambda_n||\theta||_1,\\
&&\text{subject to:     }\quad \|\theta\|_2\le r.\nonumber
\end{eqnarray}

Before presenting our main theorem, we need additional assumptions on the feature vector $x.$
\begin{assumption}\label{assume3}
 The feature vector $x$ has a probability density function in $\mathbb{R}^p.$ In addition, there exists constant $M>1$ that is independent of dimension $p$ such that $||x||_{\infty}\le M\tau$ almost sure.
\end{assumption}
The following lemma shows the uniform convergence of gradient and Hessian under the Huber's contamination model in the high-dimensional setting where $p>>n.$

\begin{lemma}\label{lemma3}
Under assumption \ref{assume1} and \ref{assume3}, there exist constants $C_1, C_2,T_0, L_0$ that depend on $r,\tau, \pi,\delta,L_{\psi},$ but independent of $n,p,$ and $g,$ such that the following hold:
\begin{description}
  \item[a] The sample directional gradient converges uniformly to the population directional gradient, along the direction $(\theta-\theta_0).$
  \begin{eqnarray}
  \vv P\left(\underset{\theta\in B_2^p(r)\setminus\{0\}}{\sup}\frac{|\langle \nabla R_n(\theta)-\nabla R(\theta), \theta-\theta_0\rangle|}{||\theta-\theta_0||_1}\le (T_0+L_0\tau)\sqrt{\frac{C_1\log (np)}{n}}\right)\ge 1-\pi.
  \end{eqnarray}
  \item[b] As $n\ge C_2s_0\log(np),$ we have
  \begin{eqnarray*}
  \vv P\left(\underset{\theta\in B_2^p(r)\cap B_2^p(s_0), \nu\in B^p_2(1)\cap B^p_0(s_0)}{\sup}|\langle \nu, \left(\nabla^2 R_n(\theta)-\nabla^2 R(\theta)\right)\nu\rangle|\le\tau^2 \sqrt{\frac{C_2s_0\log(np)}{n}}\right)\ge 1-\pi.
  \end{eqnarray*}
\end{description}
\end{lemma}

Now we are ready for our main theorem.
\begin{theorem}\label{thm3}
Assume that Assumption \ref{assume1} and Assumption \ref{assume3} hold and  the true parameter $\theta_0$ satisfies $||\theta_0||_2\le r/3$ and
$||\theta_0||_0\le s_0.$ Then there exist constants $C,C_0,C_1,C_2$ that are dependent on $(\rho, L_{\psi},\tau^2,r,\gamma, \pi)$ but independent on $(\delta,s_0,n,p,M)$ such that as $n\ge C s_0 \log p$ and $\lambda_n= C_0M\sqrt{\frac{\log p}{n}}+\frac{C_1}{\sqrt{s_0}}\delta ,$ the following hold with probability as least $1-\pi:$
 \begin{description}
        \item[(a)] All stationary points of problem (\ref{form3}) are in $B_2^p(\theta_0, \eta_0+\frac{\sqrt{s_0}}{1-\delta}\lambda_nC_2)$
  \item[(b)] As long as $n$ is large enough such that $n\ge C s_0 \log^2 p$ and the contamination ratio $\delta$ is small such that $(\eta_0+\frac{1}{1-\delta}\sqrt{s_0}\lambda_nC_2)\le \eta_1, $ the problem (\ref{form3}) has a unique local stationary point which is also the global minimizer.
\end{description}
\end{theorem}
The proof of Theorem \ref{thm3} is based on several lemmas, which are  postponed to the appendix. We believe that some of our lemmas are of interest in their own right. Theorem \ref{thm3} implies the estimation error of the penalized M-estimator is bounded as the order of $O(\delta+\sqrt{\frac{s_0\log p}{n}}),$ which achieves the minimax estimation rate \citep{chen:2016}. Moreover, it implies that the penalized M-estimator has good tractability when the percentage of outliers $\delta$ is small.


\section{Example}
In this section, we use some examples to illustrate our general theoretical results about the robustness and tractability of M-estimators. In the first subsection, we consider the low-dimensional regime and study a family of M-estimators with a specific loss function known as Huber's loss \citep{huber:1964}. In the second subsection, we consider the high-dimensional regime and study the penalized M-estimator with Welsch's exponential squared loss \citep{dennis:1978,rey:2012,wang:2013}. In both subsections, we will derive the explicit expression of the two critical radius $\eta_0,$ $\eta_1$ and discuss the robustness and tractability of the corresponding M-estimators.
\subsection{M-estimator via Huber's loss}

In this subsection, we illustrate the general results presented in Section~\ref{sec.M-est} by studying the Huber's loss function \citep{huber:1964}
\begin{eqnarray}\label{huber}
\rho_{\alpha}(t)=\left\{
                    \begin{array}{ll}
\frac{1}{2}t^2, \quad\text{if $|t|\le \alpha$}\\
\alpha(|t|-\alpha/2), \quad \text{if $|t|> \alpha.$}
 \end{array}
                   \right.
 \end{eqnarray}
where $\alpha> 0$ is a tuning parameter. The corresponding M-estimator is obtained by solving the optimization problem
\begin{eqnarray}\label{exp1}
\underset{\theta}{\min}&& \hat{R}_{n}(\theta):=\frac{1}{n} \sum_{i=1}^{n}\rho_{\alpha}(y_i-\langle\theta, x_i\rangle),\\
\text{subject to}&& ||\theta||_2\le r. \nonumber
\end{eqnarray}
First, note the loss function $\rho_{\alpha}(t)$ in (\ref{huber}) is convex. Thus, the corresponding M-estimator should be tractable even though there are some outliers. Second,  when $\alpha$ goes to $0,$  $\rho_{\alpha}(t)$  will converges to $t^2/2.$ Thus, the least square estimator is a special case of the M-estimator obtained from (\ref{exp1}), which is not robust to outliers. Third, for fixed $\alpha>0,$ $\rho_{\alpha}'(t),\rho_{\alpha}''(t)$ are all bounded. Intuitively, this implies that the impact of outlier observations of $y_i$ will be controlled and thus the corresponding statistical procedure will be robust.

We now study the robustness and tractability of the M-estimator of (\ref{exp1}) based on our framework in Theorem \ref{thm2}.
In order to emphasize on the effects of the tuning parameter $\alpha$ and the contamination ratio $\delta$ on the robustness property and tractability property, we consider a simplified assumption on the feature vector $x_i$ and the pdf of idealized residual $f_0.$
\begin{assumption}\label{assume2}
  \begin{description}
  \item[(a)] The feature vector $x_i$ are i.i.d multivariate Gaussian distribution $N(0,\tau^2I_{p\times p}).$
  \item[(b)] The idealized noise pdf $f_0(\epsilon)$ has Gaussian distribution $N(0,\sigma^2).$
  \item[(c)] Assume the true parameter $||\theta_0||_2\le r/3.$
\end{description}
\end{assumption}

\begin{corollary}\label{cor1}
Under Assumption \ref{assume2}, for any $\delta,\alpha\ge 0,$ there exist two constants $\eta_0(\delta,\alpha),\eta_1(\delta,\alpha):$
\begin{eqnarray}
  \eta_0(\delta,\alpha)&=&\frac{\delta}{1-\delta}\frac{4\sqrt{2\pi}\sigma^3}{(\alpha^2+3\sigma^2)\tau}e^{\frac{\alpha^2+22\tau^2r^2}{2\sigma^2}}\\
  \eta_1(\delta,\alpha)&=&+\infty,
  \end{eqnarray}
  such that when the number of data points $n$ is large, with high probability, any stationary points of the empirical risk function $\hat{R}_n(\theta)$ in (\ref{exp1}) belongs  in the ball $B_2^{p}(\theta_0,2\eta_0(\delta,\alpha)).$ Moreover, the empirical risk function $\hat{R}_n(\theta)$ in (\ref{exp1}) is strongly convex in the ball $B_2^p(\theta_0,\eta_1(\delta,\alpha)).$ Thus,  there exists a unique stationary point of  $\hat{R}_{n}(\theta),$ which is the corresponding M-estimator.
\end{corollary}

Note $\eta_1(\delta,\alpha)=\infty,$ which means the corresponding Huber's estimator will be tractable, no matter there are outliers or not. This is consistent with the fact that the Huber's loss function is convex.
Moreover, it is interesting to see the special case of Corollary \ref{cor1} with $\alpha = +\infty,$ which reduces to the least square estimator. As we can see, with $\delta >0,$ we have $\eta_0(\delta,\alpha=+\infty)=+\infty,$ which implies  the solution of the optimization problem in (\ref{exp1}) can be arbitrarily in the ball $B_2^{p}(0,r=10),$ even when the proportion of outliers is small. Thus it is not robust to the outliers. This recovers the well-known fact:  the least square estimator is easy to compute, but is very sensitive to outliers.

Additionally, for another special case with $\delta=0$ and $\alpha>0,$ we have $\eta_0(\delta=0,\alpha)=0,$ which means the true parameter $\theta_0$ is the unique stationary point of the risk function. This implies the Huber's estimator is consistent when there are no outliers.

\subsection{Penalized M-estimator via Welsch's exponential squared loss}

In this subsection, we illustrate the general results presented in Section~\ref{third.M-est} by considering a family of M-estimators with a specific nonconvex loss function known as Welsch's exponential squared loss \citep{dennis:1978,rey:2012,wang:2013},
\begin{eqnarray}\label{welsh}
\rho_{\alpha}(t)=\frac{1-\exp(-\alpha t^2/2)}{\alpha},
\end{eqnarray}
where $\alpha\ge 0$ is a tuning parameter. The corresponding penalized M-estimator is obtained by solving the optimization problem
\begin{eqnarray}\label{exp2}
\underset{\theta}{\min}&& \hat{L}_{n}(\theta):=\frac{1}{n} \sum_{i=1}^{n}\rho_{\alpha}(y_i-\langle\theta, x_i\rangle)+\lambda_n||\theta||_1,\\
\text{subject to}&& ||\theta||_2\le r. \nonumber
\end{eqnarray}
The non-convex loss function $\rho_{\alpha}(t)$ in (\ref{welsh}) has been used in other contexts such as robust estimation and robust hypothesis testing, see \cite{ferrari:2010,qin:2013}, as it has many nice properties. First, it is a smooth function of both $\alpha$ and $t,$ and the gradient and Hessian are well-defined. Second,  when $\alpha$ goes to $0,$  $\rho_{\alpha}(t)$  will converges to $t^2/2.$ Thus, the LASSO estimator is a special case of the M-estimator obtained from (\ref{exp2}). Third, for fixed $\alpha>0,$ $\rho_{\alpha}(t), \rho_{\alpha}'(t),\rho_{\alpha}''(t)$ are all bounded. Intuitively, this implies that the impact of outlier observations of $y_i$ will be controlled and thus the corresponding statistical procedure will be robust.

We now study the robustness and tractability of the penalized M-estimator of (\ref{exp2}) based on our framework in Theorem \ref{thm3}.  When $\alpha$ goes to $0,$ the M-estimator reduces to the LASSO estimator, which can be computed easily. However, it is also known to be very sensitive to the outliers \citep{alfons:2013}. On the other hand, when $\alpha$ increases, the estimator becomes  more robust, but may lose tractability due to the highly non-convexity of the function $\rho_{\alpha}(t)$ as well as the presence of the outliers.

In order to emphasize on the relation between the tuning parameter $\alpha$ and the contamination ratio $\delta,$ we consider a simplified assumption on the feature vector $x_i$ and the pdf of idealized residual $f_0.$
\begin{assumption}\label{assume4}
  \begin{description}
  \item[(a)] The feature vector $x_i$ are i.i.d multivariate uniform distribution $[-\tau,\tau]^p.$
  \item[(b)] The idealized noise pdf $f_0(\epsilon)$ has Gaussian distribution $N(0,\sigma^2).$
  \item[(c)] The true parameter $||\theta_0||_2\le r/3.$
\end{description}
\end{assumption}

With Assumption \ref{assume4} and  Theorem \ref{thm3}, we can get the following corollary, which characterizes the robustness and tractability of the penalized M-estimator with Welsch's exponential squared loss in (\ref{exp2}):
\begin{corollary}\label{cor2}
Assume that Assumption \ref{assume4} holds and  the true parameter $\theta_0$ satisfies $||\theta_0||_2\le r/3,$ for any $\pi\in(0,1),$ there exist a constant $C_{\pi}$ such that if choose $\lambda_n= 2C_{\pi}\tau\sqrt{\frac{\log p}{n}}+\frac{\alpha\tau}{2}\frac{\delta}{\sqrt{s_0}},$ as $n>>s_0\log p,$ the following hold with probability as least $1-\pi:$
 \begin{description}
        \item[(a)] All stationary points of problem (\ref{exp2}) are in $B_2^p(\theta_0, (1+2\tau)\eta_0)$
        \item[(b)] The empirical risk function $\hat{L}_{n}(\theta)$ are strong convex in the ball $B_2^p(\theta_0, \eta_1)$
  \item[(c)] As long as $n$ is large enough and the contamination ratio $\delta$ is small such that $\left(1+2\tau\right)\eta_0\le \eta_1, $ the problem (\ref{exp2}) has a unique local stationary point which is also the global minimizer.
\end{description}
Here
\begin{eqnarray}
 \eta_0(\delta,\alpha)&=&\frac{\delta}{1-\delta}\sqrt{\frac{e}{\alpha}}\frac{4(1+\alpha\sigma^2)^{3/2}}{\tau}e^{\frac{32\alpha r^2\tau^2}{3(1+\alpha\sigma^2)}}\\
  \eta_1(\delta,\alpha)&=&\frac{1}{3\sqrt{3\alpha}(1+\alpha\sigma^2)^{3/2}\tau}\left[\tau^2-\delta(\tau^2+(1+\alpha\sigma^2)^{3/2})\right].
 \end{eqnarray}
\end{corollary}

It is interesting to see the special case of Corollary \ref{cor2} with $\alpha = 0,$ which reduces to the LASSO estimator. On the one hand, with $\alpha = 0,$ we have $\eta_1(\delta,\alpha=0)=+\infty$  for any $\delta>0.$ This means that the corresponding risk function is strongly convex in the entire region of $B_2^{p}(0,r=10)$, and hence it is always tractable.  On the other hand,  since $\eta_0(\delta,\alpha=0)=+\infty,$   the solution of the optimization problem in (\ref{exp2}) can be arbitrarily in the ball $B_2^{p}(0,r=10),$ even when the proportion of outliers is small. Thus it is not robust to the outliers. This recovers the well-known fact:  the LASSO estimator is easy to compute, but is very sensitive to outliers.

Additionally, for another special case with $\delta=0$ and $\alpha>0,$ we have $\eta_0(\delta=0,\alpha)=0,$ which means the true parameter $\theta_0$ is the unique stationary point of the risk function. This implies the Welsch's estimator has nice tractability when there is no outliers. However, when the percentage of outlier $\delta$ is increasing, $\eta_1(\delta,\alpha)$ will decrease, which implies more outliers will reduce the tractability of the M-estimator.
\section{Simulation results}
In this section, we   report the simulation results by using Welsch's exponential loss \citep{dennis:1978} when the data are contaminated, using synthetic data setting.
We first generate covariates $x_i\sim N(0,I_{p\times p})$ and responses $y_i=\langle\theta_0, x_i\rangle+\epsilon_i,$ where $||\theta_0||_2=1.$
We consider the case when the residual term $\epsilon_i$ have gross error model with contamination ratio $\delta$, i.e., $\epsilon_i\sim (1-\delta)N(0,1)+\delta N(\mu_{i}, 3^2)$ where $\mu_i = ||x_i||^2_2+ 1.$ The outlier distribution is chosen to highlight the effects of outliers when they are dependent on $x_i$ and has non-zero mean.

In the first part, we consider the low-dimensional case when the dimension $p=10.$ Specifically, we generate $n=200$ pairs of data $(y_i,x_i)_{i=1,..,n}$ with dimension $p=10$ and with different choices of contamination ratios $\delta.$
We use projected gradient descent to solve the optimization problem in (\ref{exp1}) with  $r=10.$ In order to make the iteration points be inside the ball, we will
project the points back into $B_2^p(0,r=10)$ if they fall out of the ball. The step size is fixed as $1.$
In order to test the tractability of the M-estimator, we run gradient descent algorithm with $20$ random initial values in the ball $B_2^p(0,r=10)$ to see whether the gradient descent algorithm can converge to the same stationary point or not. Denote $\hat{\theta}(k)$ as the $k^{th}$ iteration points, Figure \ref{fig1} shows the convergence of the gradient descent algorithm for the exponential loss   with the choice of $\alpha=0.1$ under the gross error model with different $\delta.$ From Figure \ref{fig1} we observe when the proportion of outliers is small (i.e., $\delta\le0.1,$)  gradient descent could converge to the same stationary point fast.  However, when the contamination ratio $\delta$ becomes larger, gradient descent may not converge to the same point for different initial points, indicating the loss of tractability {\em for the same objective function} with   increasing proportion of outliers. Those observations are consistent to our Theorem \ref{thm2}, which asserts the  M-estimator is tractable {when the contamination ratio $\delta$ is small}.

To illustrate the robustness of the M-estimator, we generate $100$ realizations of $(Y,X)$ and run gradient descent algorithm with different initial values. The average estimation errors between the M-estimator and the true parameter $\theta_0$ are presented in Figure \ref{fig3}.
As we can see, when $\delta=0,$ all estimators have small estimation errors, which are well expected as those M-estimators are consistent without outliers \citep{huber:1964,huber:2009}. However, for the M-estimator with $\alpha=0,$ i.e., the least square estimator, the estimation error will increase dramatically as the proportion of outliers increases. This confirms that the least square estimator is not robust to the outliers.

Meanwhile, when $\alpha = 0.1,$ the overall estimation error does not increase much even with $40\%$ outliers, which clearly demonstrate the robustness of the M-estimator. Note that when $\alpha$ is further increased from $0.1$ to $0.3,$ although the estimator error is still very small for $\delta\leq 0.2,$ it will increase dramatically when $\delta$ is greater than $0.2.$  We believe that two reasons contribute to this phenomenon: robustness starts to decrease when $\alpha$ becomes too large; and more importantly, the algorithm fails to find the global optimum due to multiple stationary points when $\alpha$ is large. Thus for each $\alpha,$ there exists a critical bound of $\delta,$ such that the estimator will be robust and tractable efficiently when the proportion of outliers is smaller than that bound.

In the second part, we present our results in the high-dimensional region when $p=400$. Data $(y_i, x_i)$ are generated from the same gross error model in the previous simulation study, with the true parameter $\theta_0$ a sparse vector with $10$ nonzero entries. All nonzero entries are set to be $1/\sqrt{10}.$ We use proximal gradient descent algorithm to solve problem (\ref{form3}). Similarly, we will
project the points back into $B_2^p(0,r=10)$ if they fall out of the ball.
Figure \ref{fig8} shows the convergence of the proximal gradient descent algorithm for the nonconvex exponential loss with the choice of $\alpha=0.1$ and $L_1$ regularizer with the parameter $\lambda=0.1$ under the gross error model with different $\delta.$ From Figure \ref{fig8} we observe when the percentage of outliers is small, the algorithm will converge to the same stationary point fast, which implies there is only one unique stationary point. When $\delta$ is larger, the converge rate become slower, which implies there may exist another stationary points. Those simulation results reflect our theoretical result for the tractability of the penalized M-estimator in high-dimensional regression.

\begin{figure}[htp]
  \begin{minipage}[t]{0.49\linewidth}
  \includegraphics[width=1\linewidth,height=1.8in]{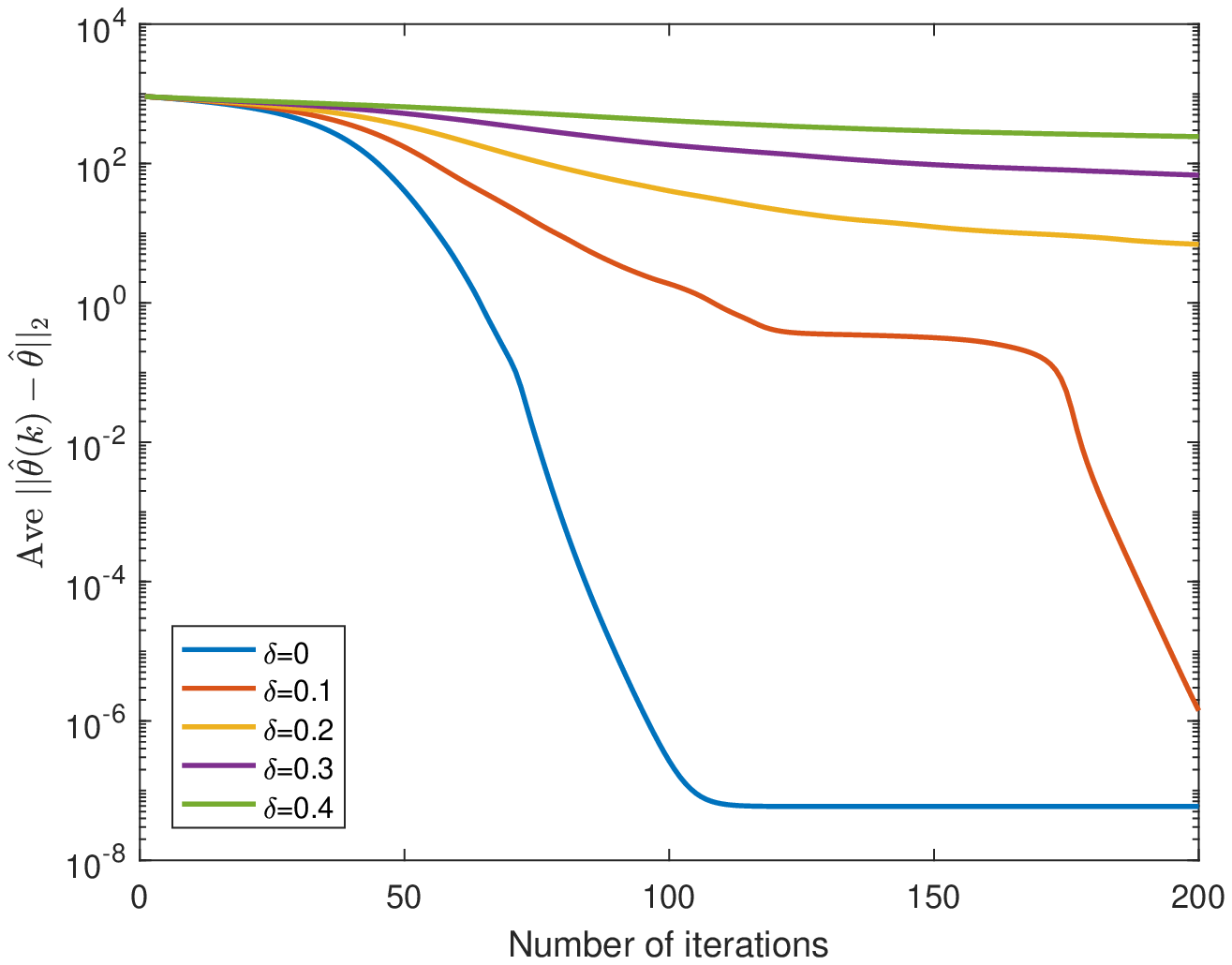}
  \caption{The convergence of gradient descent algorithm for different $\delta.$ Y-axis is with log scale.}\label{fig1}
  \end{minipage}
  \hfill
  \begin{minipage}[t]{0.49\linewidth}
  \includegraphics[width=1\linewidth,height=1.8in]{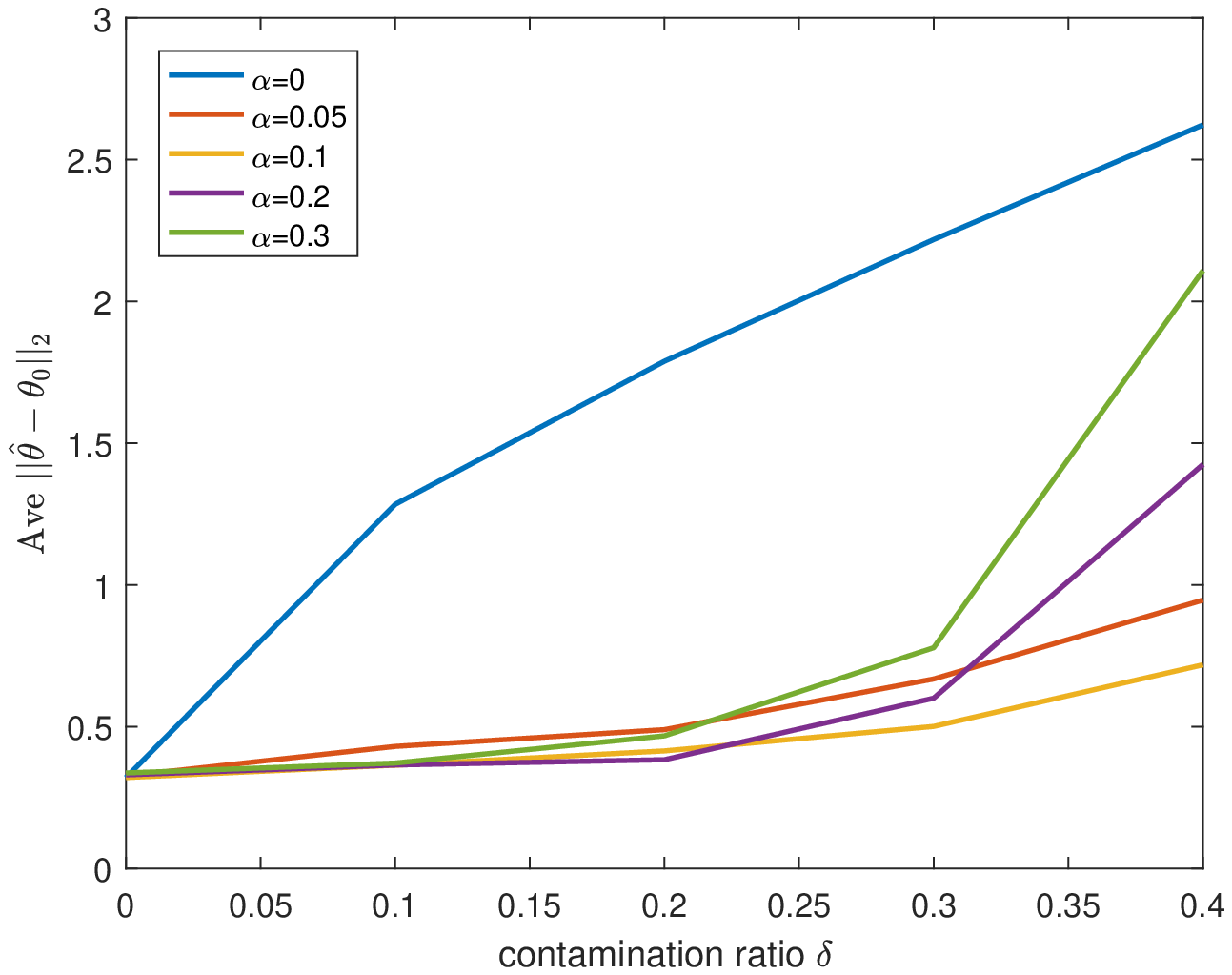}
  \caption{The estimation error for different  $\alpha$ and $\delta$}\label{fig3}
  \end{minipage}
\end{figure}

\begin{figure}[htp]
\centering
  \includegraphics[width=0.49\linewidth,height=1.8in]{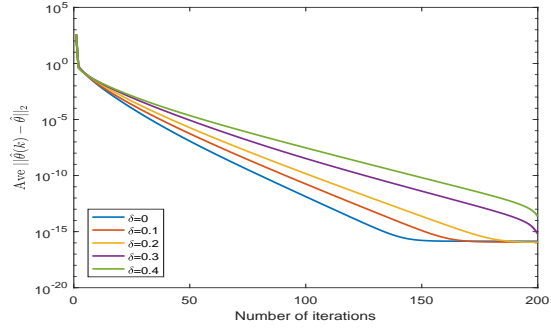}
  \caption{The convergence of gradient descent algorithm for different $\delta.$ Y-axis is with log scale.} \label{fig8}
\end{figure}

\section{Case study}
In this section, we present a case study of the robust regression problem for the Airfoil Self-Noise  dataset \citep{brooks:2014}, which is available on UCI Machine Learning Repository. The dataset was processed by NASA and
is commonly used for regression study to learn the relation between the airfoil self-noise and five explanatory variables.
Specifically, the dataset contain the following $5$ explanatory variables: Frequency (in Hertzs),
Angle of attack (in degrees),
Chord length,(in meters),
Free-stream velocity (in meters per second),
and Suction side displacement thickness (in meters).
There are $1503$ observations in the dataset. The response variable is Scaled sound pressure level (in decibels).
In this section, the five explanatory variables are scaled to have zero mean and unit variance. Then, we corrupt the response by adding noise $\epsilon$ from the same gross error model as the previous section: $\epsilon_i \sim (1-\delta)N(0,1)+\delta N(\mu_i, 3^2)$ with $\mu_i = ||x_i||^2_2+1.$

We consider the M-estimator using Welsch's exponential loss \citep{dennis:1978} on the dataset to validate the tractability and the robustness of the corresponding M-estimator.
First, we run $100$ Monte Carlo simulations. At each time, we split the dataset which consists of  $1503$ pairs of data into a training dataset of size $1000$ and a testing dataset of size $503.$ Then for the training dataset, we use gradient descent method with $20$ different initial values to update the iteration points.

Figure \ref{fig6} shows the average distance between each iteration point and the optimal point with the choice of $\alpha=0.7$ and step size  $0.5$.  Clearly, when $\delta$ is smaller than $0.3,$ gradient descent will converge to the same local minimizer, which implies the uniqueness of the stationary point. This result demonstrates the nice tractability of the M-estimator under the gross error model when the proportion of outliers is small. Then, using the optimal point as the M-estimator, we calculate the prediction error, which is the mean square error on the testing data. Figure \ref{fig7} shows the average prediction error on the testing data. As we can see, the prediction error with the choice of $\alpha=0$ will increase dramatically when the percentage of outliers increases. In contrast, the prediction errors of M-estimators with $\alpha= 0.4$ is stable even with a large percentage of outliers. This illustrates the robustness of M-estimators for some positive $\alpha$.

\begin{figure}[htp]
  \begin{minipage}[t]{0.49\linewidth}
  \includegraphics[width=1\linewidth,height=1.8in]{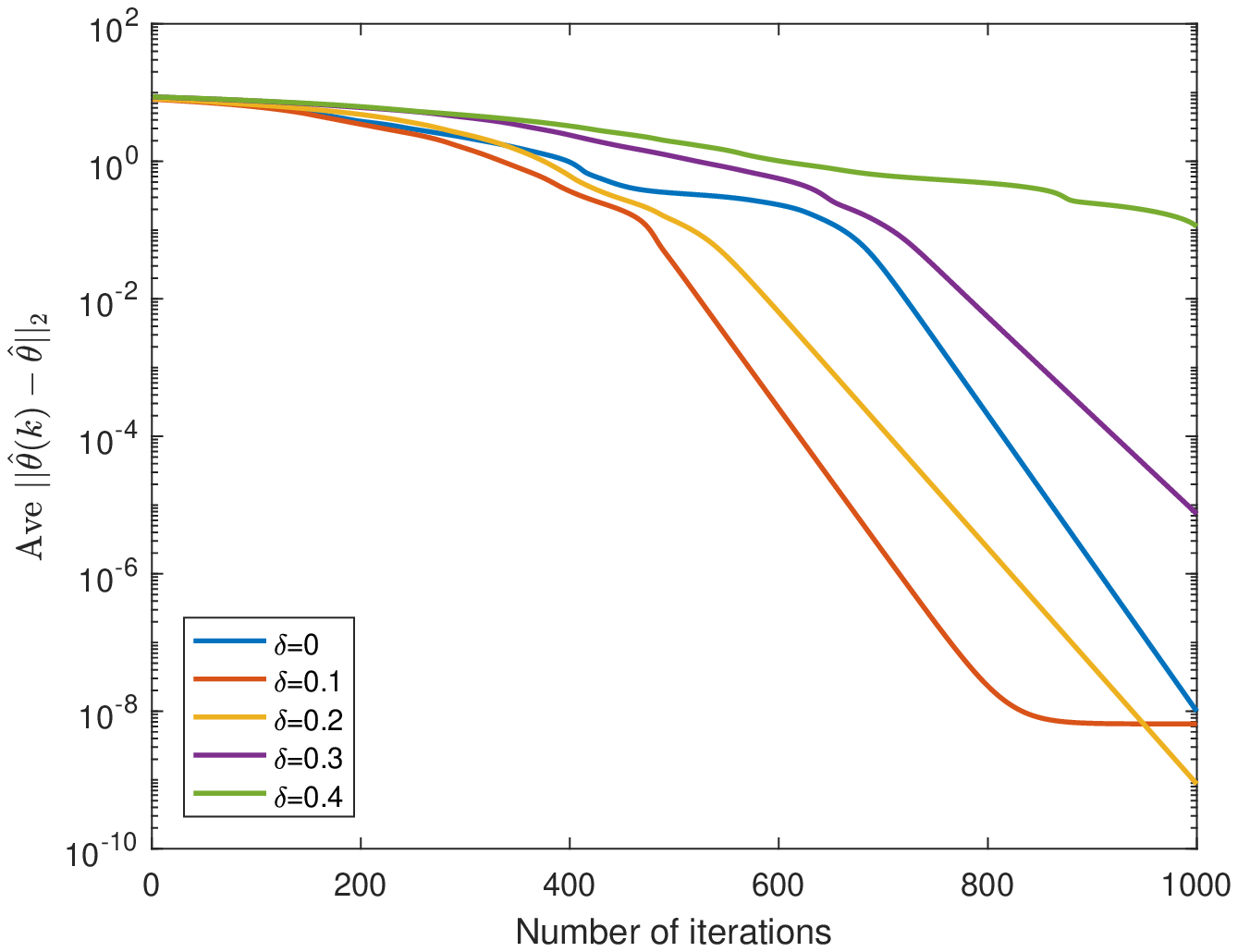}
  \caption{The convergence of gradient descent algorithm for different $\delta.$ Y-axis is with log scale.}\label{fig6}
  \end{minipage}
  \hfill
  \begin{minipage}[t]{0.49\linewidth}
  \includegraphics[width=1\linewidth,height=1.8in]{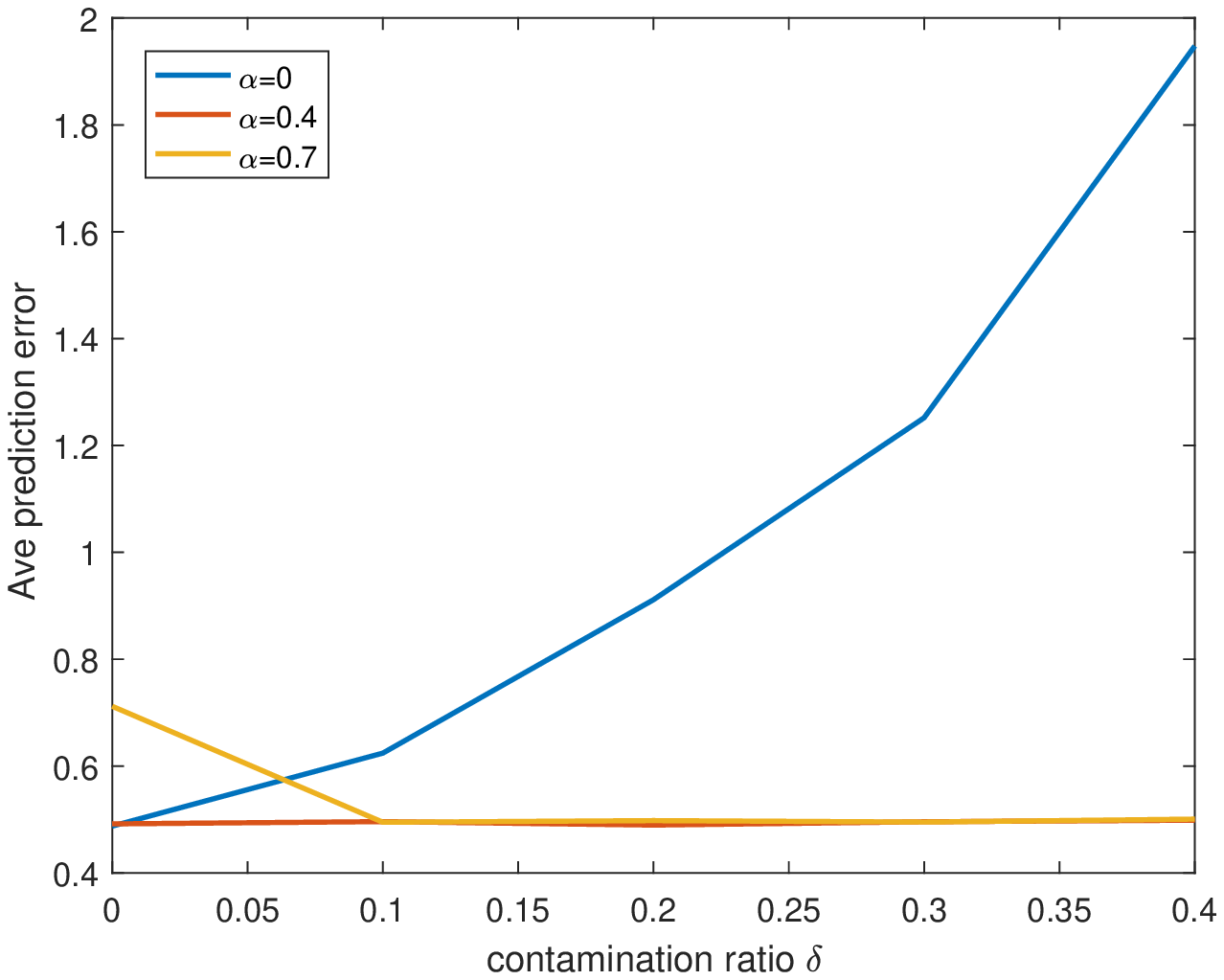}
  \caption{The prediction error for different  $\alpha$ and $\delta$}\label{fig7}
  \end{minipage}
\end{figure}
%
%
%
\section{Conclusions}

In this paper, we investigate the robustness and computational tractability of general (non-convex) M-estimators in both classical low-dimensional regime and modern high-dimensional regime. In terms of \emph{robustness}, in the low-dimensional regime, we show the estimation error of the M-estimator is as the order of $O(\delta+\sqrt{\frac{p\log n}{n}}),$ which nearly achieves the minimax lower bound of $O(\delta+\sqrt{\frac{p}{n}})$ in \cite{chen:2016}. In the high-dimensional regime, we show the estimation error of the penalized M-estimator has the estimation error as the order of $O(\delta+\sqrt{\frac{s_0\log p}{n}}),$ which achieves the minimax estimation rate \citep{chen:2016}.

In terms of \emph{tractability}, our theoretical results imply under sufficient conditions, when the percentage of arbitrary outliers is small, the general M-estimator could have good computational tractability since it has only one unique stationary point, even if the loss function is non-convex. Therefore, M-estimators can tolerate certain level of outliers by keeping both estimation accuracy and computation efficiency. Both simulation and real data case study are conducted to validate our theoretical results about the robustness and tractability of M-estimation in the presence of outliers.
\newpage

\section{Appendix}
\textbf{Proof of Lemma \ref{lemma1}}: In order to prove the uniform convergency theorem,  it is suffice to verify assumption 1, 2 and 3 in \cite{mei:2016}. Specifically, first, we will verify that the directional gradient of the population risk is sub-Gaussian (Assumption 1 in \cite{mei:2016}). Note the directional gradient of the population risk is given by $\langle \nabla \rho(Y-\langle X,\theta\rangle),\nu\rangle=\psi(Y-\langle X,\theta\rangle)\langle X,\nu\rangle.$ Since $|\psi(Y-\langle X,\theta\rangle)|\le L_{\psi},$ and  $\langle X,\nu\rangle$ is mean zero and $\tau^2$-sub-Gaussian by our assumption 1, due to Lemma 1 in \cite{mei:2016}, there exists a universal constant $C_1,$ such that $\langle \nabla \rho(Y-\langle X,\theta\rangle),\nu\rangle$ is $C_1L_{\psi}\tau^2-$sub-Gaussian. Second, we will verify that the directional Hessian of the loss is sub-exponential (Assumption 2 in \cite{mei:2016}). The directional Hessian of the loss gives $\langle \nabla^2 \rho(Y-\langle X,\theta\rangle) \nu,\nu\rangle=\psi'(Y-\langle X,\theta\rangle)\langle X,\nu\rangle^2.$ Since $|\psi'(Y-\langle X,\theta\rangle)|\le L_{\psi},$ by Lemma 1 in \cite{mei:2016}, $\langle \nabla^2 \rho(Y-\langle X,\theta\rangle) \nu,\nu\rangle$ is $C_2\tau^2$-sub-exponential. Third, let $H=||\nabla^2 R(\theta_0)||_{op}$ and
$J^*= \mathbf{E}\left[\underset{\theta_1 \neq \theta_2}{\sup}\frac{||(\psi'(\theta_1)-\psi'(\theta_2))xx^T||_{op}}{||\theta_1-\theta_2||_2}\right].$ Then, we can show $H\le L_{\psi}\tau^2$ and $J^*\le L_{\psi} (p\tau^2)^{3/2}.$ Therefore, there exists a constant $c_h$ such that $H\le \tau^2p^{c_h}$ and $J^*\le \tau^3 p^{c_{h}},$ which verifies the assumption 3 in \cite{mei:2016}. Therefore, the uniform convergency of gradient and Hessian in theorem 1 in \cite{mei:2016} holds for our gross error model.
\qed
\\

\textbf{Proof of Theorem \ref{thm1}}:
Part (a): It is suffice to show that $\langle \theta-\theta_0, \nabla R(\theta)\rangle>0$ for all $||\theta-\theta_0||_2>\eta_0.$
Note by Assumption \ref{assume1}(d), we have $h(z)=\int_{-\infty}^{+\infty}\psi(z+\epsilon)f_0(\epsilon)d\epsilon>0$ as $z>0$ and $h'(0)>0.$
Define $H(s):=\underset{0\le z\le s}{\inf}\frac{h(z)}{z},$ it is easy to see that $H(s)>0$ for all $s>0.$
Then, we have
\begin{eqnarray*}
\langle \theta-\theta_0, \nabla R(\theta)\rangle&=&\vv E \left[\vv E[\psi(z+\epsilon)z|z=\langle \theta_0-\theta,X\rangle]\right]\\
&=&(1-\delta)\vv E[h(\langle \theta-\theta_0,X\rangle)\langle \theta-\theta_0,X\rangle]+\delta \vv E \left[\vv E_{g}(\psi(z+\epsilon)z|z=\langle \theta_0-\theta,X\rangle)\right]\\
&\ge& (1-\delta)H(s) \vv E[\langle \theta-\theta_0,X\rangle^2 I_{(|\langle \theta-\theta_0,X\rangle|\le s)}]-\delta L_{\psi}\vv E |\langle \theta_0-\theta,X\rangle|\\
&=&(1-\delta)H(s)\vv E[\langle \theta-\theta_0,X\rangle^2- \langle \theta-\theta_0,X\rangle^2I_{(|\langle \theta-\theta_0,X\rangle|> s)}]-\delta L_{\psi} \vv E|\langle \theta-\theta_0,X\rangle|\\
&\ge&(1-\delta)H(s)\left[\vv E[\langle \theta-\theta_0,X\rangle^2]-\left(\vv E[\langle \theta-\theta_0,X\rangle^4]\cdot\vv P(|\langle \theta-\theta_0,X\rangle|> s)\right)^{1/2}\right]\\
&&-\delta L_{\psi}(\vv E|\langle \theta-\theta_0,X\rangle|^2)^{1/2}\\
&\overset{\text{(i)}}{\ge}& (1-\delta) H(s) ||\theta-\theta_0||_2^2\tau^2 \left(\gamma-\sqrt{c_2 \vv P(|\langle \theta-\theta_0,X\rangle|> s)}\right)-\delta L_{\psi} ||\theta-\theta_0||_2 \tau\\
&\overset{\text{(ii)}}{\ge}& (1-\delta)H(s) ||\theta-\theta_0||_2^2\tau^2 \left(\gamma-\sqrt{\frac{c_2 \vv E(|\langle \theta-\theta_0,X\rangle|^4)}{s^4}}\right)-\delta L_{\psi} ||\theta-\theta_0||_2\tau\\
&\ge& (1-\delta)H(s) ||\theta-\theta_0||_2^2\tau^2 \left(\gamma-\sqrt{\frac{c_2\cdot c_2 \tau^4 ||\theta-\theta_0||_2^4}{s^4}}\right)-\delta L_{\psi} ||\theta-\theta_0||_2\tau\\
&\ge& (1-\delta)H(s) ||\theta-\theta_0||_2^2\tau^2 \left(\gamma-\frac{c_2 \tau^2 ||\theta-\theta_0||_2^2}{s^2}\right)-\delta L_{\psi} ||\theta-\theta_0||_2\tau\\
&\ge& (1-\delta)H(s) ||\theta-\theta_0||_2^2\tau^2 \left(\gamma-\frac{16c_2 \tau^2 r^2}{9s^2}\right)-\delta L_{\psi} ||\theta-\theta_0||_2\tau.
\end{eqnarray*}
Here (i) holds from the fact that if $X$ has mean zero and is $\tau^2$-sub-Gaussian, then for all $u\in \mathbb{R}^p,$
\begin{eqnarray*}
\vv E|\langle u, X\rangle|^2&\le& ||u||^2_2\tau^2,\\
\vv E|\langle u, X\rangle|^4&\le& c_2 ||u||^4_2\tau^4,
\end{eqnarray*}
where $c_2$ is a constant \citep{boucheron:2013}. (ii) holds from Chebyshev's inequality.
Thus, a choice of $\tilde{s}=\frac{8\tau r}{3}\sqrt{\frac{c_2}{\gamma}}$ will ensure that
\begin{equation}\label{eq3}
\langle \theta-\theta_0, \nabla R(\theta)\rangle\ge (1-\delta)\frac{3}{4}H(\frac{8\tau r}{3}\sqrt{\frac{c_2}{\gamma}})||\theta-\theta_0||_2^2\tau^2 \gamma-\delta L_{\psi}||\theta-\theta_0||_2\tau ,
\end{equation}
which is greater than $0$ when
\begin{eqnarray}\label{eq4}
||\theta-\theta_0||_2>\frac{\delta L_{\psi}}{(1-\delta)\frac{3}{4}H(\frac{8\tau r}{3}\sqrt{\frac{c_2}{\gamma}})\tau\gamma}:=\eta_0.
\end{eqnarray}
Therefore, there are no stationary point outside of the ball $B_2^{p}(\theta_0,\eta_0).$

Part(b): We first look at the minimum eigenvalue of the Hessian $\nabla^2R(\theta)$ at $\theta=\theta_0.$ For any $u\in \mathbb{R}^{p}, ||u||_2=1,$
\begin{eqnarray*}
\langle u, \nabla^2R(\theta_0)u\rangle &=& (1-\delta)\vv E_{f_0} [\psi'(\epsilon)\langle X,u\rangle^2]+\delta \vv E_{g} [\psi'(\epsilon)\langle X,u\rangle^2]\\
&=& (1-\delta)\vv E_{f_0} [\psi'(\epsilon)]\vv E[\langle X,u\rangle^2] +\delta \vv E_{g} [\psi'(\epsilon)\langle X,u\rangle^2]\\
&\ge&  (1-\delta)h'(0)\gamma \tau^2-\delta L_{\psi} \tau^2.
\end{eqnarray*}
Therefore,  we have the minimum eigenvalue of $\nabla^2R(\theta_0)$ is greater than $0$ as long as $\delta<\frac{h'(0)\gamma}{h'(0)\gamma+L_{\psi}}.$\\
Then we look at the operator norm of $\nabla^2 R(\theta) -\nabla^2 R(\theta_0).$ For any $u\in \mathbb{R}^{p}, ||u||_2=1,$
\begin{eqnarray*}
|\langle u, (\nabla^2 R(\theta) -\nabla^2 R(\theta_0))u \rangle|&=& |\vv E [(\psi'(\langle X, \theta_0-\theta\rangle+\epsilon)-\psi'(\epsilon))\langle X, u\rangle^2]|\\
&=& |\vv E [\psi''(\xi)\langle X, \theta_0-\theta\rangle\langle X, u\rangle^2]|\\
&\le& \vv E|\psi''(\xi)|\vv E |\langle X, \theta_0-\theta\rangle\langle X, u\rangle^2|\\
&\le& L_{\psi}\{\vv E[\langle X, \theta_0-\theta\rangle^2]\vv E [\langle X, u\rangle^4]\}^{1/2}\\
&\le & L_{\psi}(||\theta_0-\theta||_2^2 \tau^2c_2\tau^4)^{1/2}\\
&=& L_{\psi}\sqrt{c_2}||\theta_0-\theta||_2 \tau^3.
\end{eqnarray*}
Hence, taking
\begin{eqnarray}\label{eq12}
||\theta-\theta_0||_2\le \eta_1:=\frac{ (1-\delta)h'(0)\gamma-\delta L_{\psi}}{2\sqrt{c_2}\tau L_{\psi}}
\end{eqnarray}
guarantees that $(\nabla^2 R(\theta) -\nabla^2 R(\theta_0))_{op}\le \frac{(1-\delta)h'(0)\gamma \tau^2-\delta L_{\psi} \tau^2}{2}.$ Therefore, for all $\theta\in B_2^{p}(\theta_0,\eta_1),$ we have
\begin{equation}\label{eqn02}
\lambda_{\min}(\nabla^2 R(\theta))\ge \kappa :=\frac{(1-\delta)h'(0)\gamma-\delta L_{\psi}}{2}\tau^2,
\end{equation}
which yields there is at most one minimizer of $R(\theta)$ in the ball $B_2^{p}(\theta_0,\eta_1),$  as long as $\delta<\frac{h'(0)\gamma}{h'(0)\gamma+L_{\psi}}.$

Part (c): Note $R(\theta)$ is a continuous function on $B^p_2(r).$ Thus there exists a global minimizer, denoted by $\theta^*.$  Since we have shown that there is no stationary points outside the ball $B_2^p(\theta_0,\eta_0),$  $\theta^*$ should be in the ball $B_2^p(\theta_0,\eta_0).$
Therefore, as long as $\eta_1> \eta_0,$ i.e.,
\begin{eqnarray}
\frac{ (1-\delta)h'(0)\gamma-\delta L_{\psi}}{2\sqrt{c_2}\tau L_{\psi}}>\frac{\delta L_{\psi}}{(1-\delta)\frac{3}{4}H(\frac{8\tau r}{3}\sqrt{\frac{c_2}{\gamma}})\tau\gamma},
\end{eqnarray}
there exists and only exists a unique stationary point of $R(\theta),$ which is also the global optimum $\theta^*.$ \qed

\textbf{Proof of Theorem \ref{thm2}}
Based on Lemma \ref{lemma1}, there exists a constant $C$ such that when $n\ge Cp\log p,$
\begin{eqnarray}
\vv P\left(\underset{\theta\in B^{p}(0,r)}{\sup} ||\nabla \hat{R}_{n}(\theta) -\nabla R(\theta)||_2\le \tau \delta L_{\psi}\right)\ge 1-\pi\\
\vv P\left(\underset{\theta\in B^{p}(0,r)}{\sup} ||\nabla^2 \hat{R}_{n}(\theta) -\nabla^2 R(\theta)||_{op}\le \kappa/2\right)\ge 1-\pi.
\end{eqnarray}
Part (a):
Note
\begin{eqnarray}\label{eq5}
\langle \theta-\theta_0, \nabla \widehat{R}_n(\theta)\rangle&\ge&\langle \theta-\theta_0, \nabla R(\theta)\rangle-||\nabla \hat{R}_{n}(\theta) -\nabla R(\theta)||_{2} ||\theta-\theta_0||_2\\
&\ge&(1-\delta)\frac{3}{4}H(\frac{8\tau r}{3}\sqrt{\frac{c_2}{\gamma}})||\theta-\theta_0||_2^2\tau^2 \gamma-2\tau\delta L_{\psi}||\theta-\theta_0||_2
\end{eqnarray}
which is greater than $0$ when
\begin{eqnarray}\label{eq6}
||\theta-\theta_0||_2>\frac{2\delta L_{\psi}}{(1-\delta)\frac{3}{4}L(\frac{8\tau r}{3}\sqrt{\frac{c_2}{\gamma}})\tau\gamma}=2\eta_0.
\end{eqnarray}
Therefore, there are no stationary points outside of the ball $B_2^{p}(\theta_0,2\eta_0).$

Part (b):
For the least eigenvalue of the empirical Hessian in $B_2^p(\theta_0,\eta_1),$ we have
\begin{eqnarray}\label{eq7}
 \underset{||\theta-\theta_0||_2\le \eta_1}{\inf}\lambda_{\min}(\nabla^2 \widehat{R}_n(\theta))&\ge&  \underset{||\theta-\theta_0||_2\le \eta_1}{\inf}\lambda_{\min}(\nabla^2 R(\theta))-\underset{\theta\in B^{p}(0,\eta_1)}{\sup} ||\nabla^2 \hat{R}_{n}(\theta) -\nabla^2 R(\theta)||_{op} \nonumber\\
 &\ge& \kappa-\kappa/2=\kappa/2>0.
\end{eqnarray}
This lead to the conclusion that, $\widehat{R}_n(\theta)$ is strong convex inside the ball $B_2^{p}(\theta_0,\eta_1).$\\

Part(c):
When $2\eta_0<\eta_1,$ by strong convexity of $\widehat{R}_n(\theta)$ in $B_2^p(\theta_0,\eta_1),$  there exists a unique local minimizer, which is in $B_2^p(\theta_0,2\eta_0).$ We denote the unique local minimizer as $\widehat{\theta}_n.$

By Theorem \ref{thm1}, there is a unique stationary point of the population risk function $R(\theta)$ in the ball $B_2^p(\theta_0, \eta_0).$ Suppose $\theta^*$ is the unique stationary point of $R(\theta).$ By Taylor expansion of $\widehat{R}_n(\theta)$ at the point $\theta^*$, there exists a $\tilde{\theta}$ in $B^p(\theta_0, 2\eta_0),$ such that
\begin{eqnarray}
\widehat{R}_n(\widehat{\theta}_n)=\widehat{R}_n(\theta^*)+\langle \widehat{\theta}_n-\theta^*, \nabla \widehat{R}_n(\theta^*)\rangle+\frac{1}{2}(\widehat{\theta}_n-\theta^*)'\nabla^2 \widehat{R}_n(\tilde{\theta})(\widehat{\theta}_n-\theta^*)\le \widehat{R}_n(\theta^*).
\end{eqnarray}
Since by equation (\ref{eq7}), the least eigenvalue of $\nabla^2 \widehat{R}_n(\tilde{\theta})$ is greater than $\kappa/2,$ which lead to
\begin{eqnarray}
\frac{\kappa}{4} ||\widehat{\theta}_n-\theta^*||^2_2\le \langle \theta^*-\widehat{\theta}_n, \nabla \widehat{R}_n(\theta^*)\rangle \le ||\theta^*-\widehat{\theta}_n||_2 ||\nabla \widehat{R}_n(\theta^*)||_2,
\end{eqnarray}
which yield
\begin{eqnarray}\label{eq9}
||\widehat{\theta}_n-\theta^*||_2\le \frac{4}{\kappa}||\nabla \widehat{R}_n(\theta^*)||_2.
\end{eqnarray}
By Theorem \ref{thm1}, $||\theta_0-\theta^*||_2<\eta_0,$ combined with equation (\ref{eq9}) and the uniform convergency theorem in Lemma \ref{lemma1} yield
\begin{eqnarray}\label{eq10}
||\widehat{\theta}_n-\theta_0||_2\le \eta_0+\frac{4\tau}{\kappa}\sqrt{\frac{C*p\log n}{n}}.
\end{eqnarray}
\qed

\textbf{Proof of lemma \ref{lemma3}:} From the Theorem 3 in \cite{mei:2016}, the uniform convergency theorem of our Lemma \ref{lemma3} holds if Assumption 4, 5 in \cite{mei:2016} hold under the contaminated model with outliers. Here we will show
under our assumption \ref{assume1} and \ref{assume3}, there exist constants $T_0$ and $L_0$ such that
\begin{description}
  \item[a] For all $\theta\in B_2^p(r),$  $Y\in \mathbb{R}, X \in \mathbb{R}^p,$ $||\nabla_{\theta} \rho (Y-\langle X,\theta\rangle)||_{\infty}\le T_0 M$
  \item[b] There exist functions $h_1: \mathbb{R} \times \mathbb{R}^{p+1} \rightarrow \mathbb{R},$ and $h_2: \mathbb{R}^{p+1} \rightarrow \mathbb{R}^p,$ such that
  \begin{eqnarray}
  \langle \nabla_{\theta} \rho(Y-\langle X, \theta \rangle),\theta-\theta_0\rangle=h_1(\langle \theta-\theta_0,h_2(Y,X)\rangle),Y,X).
  \end{eqnarray}
  In addition, $h_1(t,Y,X)$ is $L_0M$- Lipschitz to its first argument $t,$ $h_1(0,Y,X)=0,$ and $h_2(Y,X)$ is mean-zero and $\tau^2$-sub-Gaussian.
\end{description}

Part (a). The gradient of the loss is
\begin{eqnarray}
\nabla_{\theta}\rho(Y-\langle X,\theta\rangle)=-\psi(Y-\langle X,\theta\rangle)X.
\end{eqnarray}
By assumption \ref{assume1}, we have $|-\psi(Y-\langle X,\theta\rangle)|\le L_{\psi}.$ By assumption \ref{assume3}, we have $||X||_{\infty}\le M\tau.$ Therefore, (a) is satisfied with parameter $T_0=L_{\psi}\tau.$\\
Part (b). Note
  \begin{eqnarray}
  \langle \nabla_{\theta} \rho(Y-\langle X, \theta \rangle),\theta-\theta_0\rangle=-\psi(Y-\langle X,\theta\rangle)\langle X, \theta-\theta_0\rangle.
  \end{eqnarray}
We take $h_2(Y,X)=X,$ $t=\langle X, \theta-\theta_0\rangle$ and $h_1(t,Y,X)=-\psi(Y-t-\langle X,\theta_0 \rangle)t.$
Clearly, we have $h_1(0,Y,X)=0$ and $h_2(Y,X)$ is mean $0$ and $\tau^2$-sub-Gaussian. Furthermore, note $|t|\le 2rM\tau,$ we have
\begin{eqnarray}
|\frac{\partial}{\partial t}h_1(t,Y,X)|&=&|\psi'(Y-t-\langle X,\theta_0 \rangle)t-\psi(Y-t-\langle X,\theta_0 \rangle)|\\
&\le& 2M L_{\psi}r\tau+L_{\psi}\\
&\le& (2L_{\psi}r\tau+L_{\psi})M.
\end{eqnarray}
Therefore, $h_1(t,X,Y)$ is at most $(2L_{\psi}r\tau+L_{\psi})M$-Lipschitz in its first argument $t.$
By part (a) and part (b), we can see assumption 4, 5 are satisfied under the gross error model, which prove the uniform convergency theorem in our Lemma \ref{lemma3}.
\qed

\textbf{Proof of theorem \ref{thm3}:} We decompose the proof into four technical lemmas. First, in Lemma \ref{lemma4}, we prove there cannot be any stationary points of the regularized empirical risk $\hat{L}_n$ in (\ref{form3}) outside the region $\mathbb{A},$ which is a cone with $\mathbb{A}=\{\theta_0+\Delta: ||\Delta_{S_0^c}||_1\le 3||\Delta_{S_0}||_1\}.$ Then in Lemma \ref{lemma5}, we show there cannot be any stationary points outside the region $B_2^p(\theta_0,r_s)$ where $r_s$ is the statistical radius which is not less than $\eta_0$ in Theorem \ref{thm1}.  In Lemma \ref{lemma6}, we argue that all stationary points should have support size less or equal to $cs_0\log p.$ Finally, in Lemma \ref{lemma7}, we show there cannot be two stationary points in $B_2^p(\theta_0,\eta_1)\cap \mathbb{A}.$ Note $\hat{L}_n(\theta)$ is a continuous function, which indicates the existence of the global minimizer. Therefore, we can conclude there is and only is one unique stationary point of the regularized empirical risk $\hat{L}_n$ as long as $r_s<\eta_1.$

To start with those lemmas, we define the subgradient of $\hat{L}_n$ at $\theta$ as:
\begin{eqnarray}
\partial \hat{L}_n(\theta)=\left\{\nabla R_n(\theta) +\lambda_n \nu: \nu \in \partial ||\theta||_1\right\}.
\end{eqnarray}
Therefore, the optimality condition implies that $\theta$ is a stationary point of $\hat{L}_n$ if and only if $\mathbf{0}\in \partial \hat{L}_n(\theta).$
To simplify notations, all constants in the following lemmas are dependent on $(\rho, L_{\psi},\tau^2,r,\gamma, \pi)$ but independent on $\delta,s_0,n,p,M.$
\begin{lemma}\label{lemma4}
Let $S_0=supp(\theta_0)$ and $s_0=|S_0|.$ Define a cone $\mathbb{A}=\{\theta_0+\Delta: ||\Delta_{S_0^c}||_1\le 3||\Delta_{S_0}||_1\}\subseteq \mathbb{R}^p.$ For any $\pi>0,$ there exist constants $C_0,$ $C_1$ such that letting $\lambda_n\ge C_0M\sqrt{\frac{\log p}{n}}+\delta \frac{C_1}{\sqrt{s_0}},$ with probability at least $1-\pi,$ $\hat{L}_n(\theta)$ has no stationary points in $B_2^p(0,r)\cap \mathbb{A}^c:$
\begin{eqnarray}
  \langle z(\theta),\theta-\theta_0\rangle> 0, \quad \forall \theta\in B_2^p(0,r)\cap \mathbb{A}^c, z(\theta)\in \partial \hat{L}_n(\theta)
\end{eqnarray}

%
\end{lemma}
\begin{proof}
For any $z(\theta)\in \partial \hat{L}_n(\theta),$ it can be written as $z(\theta)=\nabla \hat{R}_n(\theta)+\lambda_n \nu (\theta),$ where $\nu(\theta)\in \partial||\theta||_1.$ Therefore, we have
\begin{eqnarray}\label{eq2}
 \langle z(\theta),\theta-\theta_0\rangle= \langle \nabla R(\theta),\theta-\theta_0\rangle+ \langle \nabla \hat{R}_n(\theta)-\nabla R(\theta),\theta-\theta_0\rangle+\lambda_n \langle \nu(\theta),\theta-\theta_0\rangle
\end{eqnarray}
Note by (\ref{eq3}) we have
\begin{equation}\label{eq13}
\langle \theta-\theta_0, \nabla R(\theta)\rangle\ge (1-\delta)\frac{3}{4}H(\frac{8\tau r}{3}\sqrt{\frac{c_2}{\gamma}})||\theta-\theta_0||_2^2\tau^2 \gamma-\delta L_{\psi}||\theta-\theta_0||_2\tau.
\end{equation}
By lemma \ref{lemma3}, for any $\pi>0,$ there exists a constant $C_{\pi}$ such that
\begin{eqnarray}\label{eq14}
\vv P(\underset{0<||\theta||_2<r}{\sup}\frac{|\langle \nabla \hat{R}_n(\theta)-\nabla R(\theta),\theta-\theta_0\rangle|}{||\theta-\theta_0||_1}\le C_{\pi}M\sqrt{\frac{\log p}{n}})>1-\pi.
\end{eqnarray}

Letting $\Delta=\theta-\theta_0,$ we have
\begin{eqnarray}\label{eq15}
 \langle \nu(\theta),\theta-\theta_0\rangle=  \langle \nu(\theta)_{S_0^c},\Delta_{S_0^c}\rangle+\langle \nu(\theta)_{S_0},\Delta_{S_0}\rangle\ge ||\Delta_{S_0^c}||_1-||\Delta_{S_0}||_1
\end{eqnarray}
Plugging (\ref{eq13}),(\ref{eq14}),(\ref{eq15}) into (\ref{eq2}) yields
\begin{eqnarray}
 \langle z(\theta),\theta-\theta_0\rangle&\ge& (1-\delta)\frac{3}{4}H(\frac{8\tau r}{3}\sqrt{\frac{c_2}{\gamma}})||\theta-\theta_0||_2^2\tau^2 \gamma-\delta L_{\psi}||\theta-\theta_0||_2\tau\\
 &-&C_{\pi}M\sqrt{\frac{\log p}{n}}(||\Delta_{S_0^c}||_1+||\Delta_{S_0}||_1)+\lambda_n(||\Delta_{S_0^c}||_1-||\Delta_{S_0}||_1).
\end{eqnarray}
Let $\lambda_n\ge 2C_{\pi}M\sqrt{\frac{\log p}{n}}+C_2,$ we have
\begin{eqnarray}\label{eq8}
 \langle z(\theta),\theta-\theta_0\rangle&\ge& (1-\delta)\frac{3}{4}H(\frac{8\tau r}{3}\sqrt{\frac{c_2}{\gamma}})||\theta-\theta_0||_2^2\tau^2 \gamma-\delta L_{\psi}||\theta-\theta_0||_2\tau \nonumber\\
 &+&C_{\pi}M\sqrt{\frac{\log p}{n}}(||\Delta_{S_0^c}||_1-3||\Delta_{S_0}||_1)+C_2(||\Delta_{S_0^c}||_1-||\Delta_{S_0}||_1).
\end{eqnarray}
Next, we will find the lower bound of $||\Delta_{S_0^c}||_1-||\Delta_{S_0}||_1$  under the constraint of $||\Delta_{S_0^c}||_1-3||\Delta_{S_0}||_1\ge0.$
 Note by Cauchy inequality, we have
 \begin{eqnarray}
 ||\Delta||^2_2\ge \frac{||\Delta_{S_0^c}||^2_1}{p-s_0}+\frac{||\Delta_{S_0}||^2_1}{s_0}
 \end{eqnarray}
 Therefore, under the constraint of $||\Delta_{S_0^c}||_1-3||\Delta_{S_0}||_1\ge0,$ the minimal value of $||\Delta_{S_0^c}||_1-||\Delta_{S_0}||_1$ is obtained when $||\Delta_{S_0^c}||_1-3||\Delta_{S_0}||_1=0$ and $||\Delta||^2_2=\frac{||\Delta_{S_0^c}||^2_1}{p-s_0}+\frac{||\Delta_{S_0}||^2_1}{s_0}.$ By solving the two equations yield
 \begin{eqnarray}
 ||\Delta_{S_0^c}||_1&=&3\sqrt{\frac{(p-s_0)s_0}{8s_0+p}}||\Delta||_2\\
  ||\Delta_{S_0}||_1&=&\sqrt{\frac{(p-s_0)s_0}{8s_0+p}}||\Delta||_2
 \end{eqnarray}
 and $||\Delta_{S_0^c}||_1-||\Delta_{S_0}||_1\ge 2\sqrt{\frac{(p-s_0)s_0}{8s_0+p}}||\Delta||_2.$
 Combined with (\ref{eq8}), setting $C_1=\frac{L_{\psi}\tau}{2}$ and $C_2=C_1\frac{\delta}{\sqrt{s_0}}$ yield $2\sqrt{\frac{(p-s_0)s_0}{8s_0+p}}C_2\ge \delta L_{\psi}\tau,$  which implies $\langle z(\theta),\theta-\theta_0\rangle>0,$ as long as $\theta\in \mathbb{A}^c,$ i.e., $||\Delta_{S_0^c}||_1-3||\Delta_{S_0}||_1>0.$
\end{proof}

\begin{lemma}\label{lemma5}
For any $\pi>0,$ $\theta\in \mathbb{A}, z(\theta)\in \partial \hat{L}_n(\theta),$ there exist constants $C_0,$ $C_1$ such that with probability at least $1-\pi,$
  \begin{eqnarray}
  \langle z(\theta),\theta-\theta_0\rangle> 0
  \end{eqnarray}
as long as $||\theta-\theta_0||_2>r_s,$ where
\begin{eqnarray}\label{rs}
r_s=\frac{\delta}{1-\delta}C_0+\frac{4\sqrt{s_0}}{1-\delta}(M\sqrt{\frac{\log p}{n}}+\lambda_n)C_1.
\end{eqnarray}
\end{lemma}
 \begin{proof}
 Since for any $\theta\in \mathbb{A},$ we have $||\theta-\theta_0||_1\le 4\sqrt{s_0}||\theta-\theta_0||_2.$ Combining with (\ref{eq2}) yields
 \begin{eqnarray}\label{eq16}
 \langle z(\theta),\theta-\theta_0\rangle&\ge& \langle \nabla R(\theta),\theta-\theta_0\rangle-C_{\pi}M\sqrt{\frac{\log p}{n}}||\theta-\theta_0||_1-\lambda_n ||\theta-\theta_1||_1\\
 &\ge& (1-\delta)\frac{3}{4}H(\frac{8\tau r}{3}\sqrt{\frac{c_2}{\gamma}})||\theta-\theta_0||_2^2\tau^2 \gamma-\delta L_{\psi}||\theta-\theta_0||_2\tau\\
 &&-(C_{\pi}M\sqrt{\frac{\log p}{n}}+\lambda_n)4\sqrt{s_0}||\theta-\theta_0||_2,
\end{eqnarray}
which is greater than $0$ as long as
\begin{eqnarray}\label{eqn20}
||\theta-\theta_0||_2\ge \frac{\delta L_{\psi}+(C_{\pi}M\sqrt{\frac{\log p}{n}}+\lambda_n)4\sqrt{s_0}}{(1-\delta)\frac{3}{4}H(\frac{8\tau r}{3}\sqrt{\frac{c_2}{\gamma}})\tau\gamma}:=r_s.
\end{eqnarray}
 Taking $C_0=\frac{L_{\psi}}{\frac{3}{4}H(\frac{8\tau r}{3}\sqrt{\frac{c_2}{\gamma}})\tau\gamma}$ and $C_1=\frac{\max(1,C_{\pi})}{\frac{3}{4}H(\frac{8\tau r}{3}\sqrt{\frac{c_2}{\gamma}})\tau\gamma}$ give the result of $r_s$ in equation (\ref{rs}).
 \end{proof}
 \begin{lemma}\label{lemma6}
 If $\delta\le 1/2,$ for any $\pi,$ there exist constants $C_0, C_1, C$ such that letting $\lambda_n\ge C_0M\sqrt{(\log p)/n}+\delta C_1/\sqrt{s_0},$ with probability at least $(1-\pi),$ any stationary points of $\hat{L}_n(\theta)$ in $ B_2^p(\theta_0,r_s)\cap \mathbb{A}$ has support size $|S(\hat{\theta})|\le C s_0\log p.$
 \end{lemma}
 \begin{proof}
 Let $\hat{\theta}\in B_2^p(\theta_0,r_s)\cap \mathbb{A}$ be a stationary point of $\hat{L}_n(\theta)$ in (\ref{form3}).  Then we have
 \begin{eqnarray}
 \nabla R_n(\hat{\theta})+\lambda_n\nu(\hat{\theta})=0,
 \end{eqnarray}
 where $\nu(\hat{\theta})\in ||\hat{\theta}||_1.$ Thus, we have
 \begin{eqnarray}\label{eq17}
 \left(\nabla R_n(\hat{\theta})\right)_j=\pm \lambda_n, \quad \forall j\in S(\hat{\theta})
  \end{eqnarray}
 Note $|\psi(y_i-\langle x_i,\theta_0\rangle)|\le L_{\psi}$ and $\langle x_i,e_j\rangle$ is $\tau^2$-subgaussian with mean $0.$ Then there exists an absolute constant $c_0$ such that   $\psi(y_i-\langle x_i,\theta_0\rangle)\langle x_i,e_j\rangle$ is $c_0L_{\psi}^2\tau^2$-subgaussian, see Lemma 1(d) in \cite{mei:2016}.
 Thus we have  $\frac{1}{n}\sum_{i=1}^{n}\psi(y_i-\langle x_i,\theta_0\rangle)\langle x_i,e_j\rangle$ is $c_0L_{\psi}^2\tau^2/n$-subgaussian with mean $\langle \nabla R(\theta_0), e_j\rangle.$
 Moreover, note $|\langle \nabla R(\theta_0), e_j\rangle|=|\delta \vv E_g \psi(y_i-\langle x_i,\theta_0\rangle)\langle x_i,e_j\rangle|\le \delta L_{\psi}\vv E|\langle x_i,e_j\rangle|\le \delta L_{\psi}\tau,$ we have for any $t>0,$
 \begin{eqnarray}
 &&\vv P(|\frac{1}{n}\sum_{i=1}^{n}\psi(y_i-\langle x_i,\theta_0\rangle)\langle x_i,e_j\rangle|\ge t+\delta L_{\psi}\tau)\nonumber\\
 &\le& \vv P(|\frac{1}{n}\sum_{i=1}^{n}\psi(y_i-\langle x_i,\theta_0\rangle)\langle x_i,e_j\rangle-\langle \nabla R(\theta_0), e_j\rangle|\ge t)
 \le 2 \exp(-\frac{t^2n}{2c_0L^2_{\psi}\tau^2}).
 \end{eqnarray}
 Thus, we can get
 \begin{eqnarray}
 \vv P\left(||\nabla R_n(\theta_0)||_{\infty}> t+\delta L_{\psi}\tau\right)&\le& p \max_{1\le j\le p} \vv P\left(|\frac{1}{n}\sum_{i=1}^{n}\psi(y_i-\langle x_i,\theta_0\rangle)\langle x_i,e_j\rangle| >t+\delta L_{\psi}\tau\right)\nonumber\\
 &\le& 2p \exp(-\frac{t^2n}{2c_0L^2_{\psi}\tau^2}).
 \end{eqnarray}
 Thus, a choice of $t=L_{\psi}\tau\sqrt{\frac{2c_0(\log p+\log 6/\pi)}{n}}$ and $C=\sqrt{c_0\log 6/\pi}$ will guarantee that
 \begin{eqnarray}\label{con1}
 \vv P\left(||\nabla\hat{R}_n(\theta_0)||_{\infty}>  L_{\psi}\tau(C\sqrt{\frac{\log p}{n}}+\delta)\right)&\le \pi/3&
 \end{eqnarray}
 Let $\lambda_n\ge2L_{\psi}\tau(C\sqrt{\frac{\log p}{n}}+\delta),$ we have the event $\left(||\nabla R_n(\theta_0)||_{\infty}<\lambda_n/2\right)$ happens with the probability at least $1-\pi/3.$ Under this event, combing with (\ref{eq17}) yields
 \begin{eqnarray}
 \lambda_n/2\le \left|\left( \nabla R_n(\theta_0)-\nabla R_n(\hat{\theta})\right)_j\right|, \quad \forall j\in S(\hat{\theta}).
 \end{eqnarray}
 Squaring and summing over $j\in S(\hat{\theta}),$ we have
 \begin{eqnarray}
 \lambda_n^2|S(\hat{\theta})|&\le& 4\left\|\left( \nabla\hat{R}_n(\theta_0)-\nabla\hat{R}_n(\hat{\theta})\right)_{S(\hat{\theta})}\right\|_2^2\\
 &=&4\left\|\left(\frac{1}{n}\sum_{i=1}^n\left(\psi(y_i-\langle \theta_0,x_i\rangle)-\psi(y_i-\langle \hat{\theta},x_i \rangle)\right)x_i\right)_{S(\hat{\theta})}\right\|_2^2\\
 &=&4\left\|\left(\frac{1}{n}\sum_{i=1}^n\left(\psi'(y_i-\langle \beta_i,x_i\rangle)\right)\langle \theta_0-\hat{\theta}, x_i\rangle x_i\right)_{S(\hat{\theta})}\right\|_2^2\\
 &\le &4L_{\psi}^2\left\|\left(\frac{1}{n}\sum_{i=1}^n\langle \theta_0-\hat{\theta}, x_i\rangle x_i\right)_{S(\hat{\theta})}\right\|_2^2
 \end{eqnarray}
 where $\beta_i$ are located on the line between $\theta_0$ and $\hat{\theta}$ obtained by intermediate value theorem.
 Moreover, by Minkowski inequality and Cauchy-Schwarz inequality yield
 \begin{eqnarray}\label{eqn01}
 \left\|\left(\frac{1}{n}\sum_{i=1}^n\langle \theta_0-\hat{\theta}, x_i\rangle x_i\right)_{S(\hat{\theta})}\right\|_2&\le&  \frac{1}{n}\sum_{i=1}^n|\langle \theta_0-\hat{\theta}, x_i\rangle| \left\|\left(x_i\right)_{S(\hat{\theta})}\right\|_2 \nonumber\\
 &\le&\frac{1}{n} \left((\sum_{i=1}^n|\langle \theta_0-\hat{\theta}, x_i\rangle| ^2)(\sum_{i=1}^n \|\left(x_i\right)_{S(\hat{\theta})}\|^2_2)\right)^{1/2}
 \end{eqnarray}

 Due to the restricted smoothness property of the sub-Gaussian random variables \cite{mei:2016}, there exists a constant $c_1$ depending on $\pi$ such that with probability at least $1-\pi/3,$ as $n\ge c_1 s_0\log p,$ we have
 \begin{eqnarray}
 \underset{\theta\in \mathbb{A}}{\sup}\frac{\frac{1}{n}(\sum_{i=1}^n|\langle \theta_0-\theta, x_i\rangle| ^2)}{||\theta-\theta_0||^2_2}\le 3\tau^2.
 \end{eqnarray}

%

 Therefore, with probability at least $1-\pi/3,$ we have
 \begin{eqnarray}\label{con2}
  \underset{\theta\in \mathbb{A}\cap B^p(\theta_0,r_s)}{\sup}\frac{1}{n}(\sum_{i=1}^n|\langle \theta_0-\hat{\theta}, x_i\rangle| ^2)\le 3\tau^2 \underset{\theta\in \mathbb{A}\cap B^p(\theta_0,r_s)}{\sup}||\theta-\theta_0||^2_2\le 3\tau^2r_s^2.
 \end{eqnarray}
 Moreover, by Lemma 13 in \cite{mei:2016}, for any $\pi,$ there exists constant $c_2$ depending on $\pi$ such that
 \begin{eqnarray}\label{con3}
 \vv P(\frac{1}{n}\sum_{i=1}^n \|\left(x_i\right)_{S(\hat{\theta})}\|^2_2>c_2\tau^2\log p)\le \pi/3.
 \end{eqnarray}

 By (\ref{con1},\ref{con2},\ref{con3}), as well as (\ref{eqn01}), at least $1-\pi,$
 \begin{eqnarray}
 \lambda_n^2|S(\hat{\theta})|&\le& 4L_{\psi}^23\tau^2r_s^2c_2\tau^2\log p\\
 &=& C r_s^2\log p
 \end{eqnarray}
 By equation (\ref{rs}) we have
 \begin{eqnarray}
 r_s^2\le C_0 (\frac{\delta}{1-\delta})^2+\frac{s_0}{(1-\delta)^2}(M^2\frac{\log p}{n}+\lambda_n^2)C_1
 \end{eqnarray}
 Taking $\lambda_n\ge C_2M\sqrt{(\log p)/n}+C_3\delta/\sqrt{s_0}$ gives us
 \begin{eqnarray}
| S(\hat{\theta})|&\le& (C_4\frac{s_0}{(1-\delta)^2}+s_0C_5)\log p\\
 &=& C s_0\log p
  \end{eqnarray}


 \end{proof}

 \begin{lemma}\label{lemma7}
 For any positive constants $C_0$ and  $\pi,$ letting $r_0=C_0s_0\log p,$ there exist constant $C_1$ such that when $n\ge C_1 s_0\log^2 p,$
 \begin{eqnarray}
 \vv P(\underset{\theta\in B^p_2(\theta_0,r)\cap B^p_0(0,r_0)}{\sup}\underset{\nu\in B^p_2(0,1)\cap B^p_0(0,r_0)}{\sup} \langle \nu,(\nabla^2 \hat{R}_n(\theta)-\nabla^2 R(\theta))\nu \rangle \le \kappa/2)\ge 1-\pi.
 \end{eqnarray}
 Moreover, the regularized empirical risk $\hat{L}_n(\theta)$ in (\ref{form3}) cannot have two stationary points in the region $B^p_2(\theta_0,\eta_1)\cap B^p_0(0,r_0/2).$
 \end{lemma}
 \begin{proof}
 According to (\ref{eqn02}), we have
 \begin{eqnarray}
 \underset{\theta\in B^p_2(\theta_0,\eta_1)}{\inf}\lambda_{\min}(\nabla^2 R(\theta))\ge \kappa.
 \end{eqnarray}
 By lemma \ref{lemma3}, there exists constant $C$ such that when $n\ge Cs_0\log^2 p,$
 \begin{eqnarray}
 \vv P\left(\underset{\theta\in B^p_2(\theta_0,\eta_1)\cap B^p_0(0,r_0)}{\inf}\underset{\nu\in B^p_2(0,1)\cap B^p_0(0,r_0)}{\inf} \langle \nu,(\nabla^2 \hat{R}_n(\theta))\nu \rangle \ge \kappa/2\right)\le \pi.
 \end{eqnarray}
 Suppose $\theta_1, \theta_2$ are two distinct stationary points of $\hat{L}_n(\theta)$ in $B^p_2(\theta_0,\eta_1)\cap B^p_0(0,r_0/2).$ Define
 $u=\frac{\theta_2-\theta_1}{||\theta_1-\theta_2||_2}.$ Since $\theta_1$ and $\theta_2$ are $r_0/2$-sparse, $u$ is $r_0$ sparse, as well as $\theta_1+tu$ for any $t\in \mathbb{R}.$ Therefore,
 \begin{eqnarray}\label{eq18}
 \langle \nabla \hat{R}_n(\theta_2), u \rangle &=&  \langle \nabla \hat{R}_n(\theta_1), u \rangle +\int_{0}^{||\theta_1-\theta_2||_2}\langle u,\nabla^2 \hat{R}_n(\theta_1+tu)u \rangle dt \nonumber\\
 &\ge&\langle \nabla \hat{R}_n(\theta_1), u \rangle +\frac{\kappa}{2}||\theta_2-\theta_1||_2.
 \end{eqnarray}
 Note the regularization term $\lambda_n||\theta||_1$ is convex, we have for any subgradients $\nu(\theta_1)\in \partial ||\theta_1||_1,$ $\nu(\theta_2)\in \partial ||\theta_2||_1,$
 \begin{eqnarray}\label{eq19}
 \lambda_n \langle \nu(\theta_2),u\rangle \ge \lambda_n \langle \nu(\theta_1), u \rangle.
 \end{eqnarray}
 Adding (\ref{eq18}) with (\ref{eq19}) gives
 \begin{eqnarray}
   \langle \nabla\hat{R}_n(\theta_2)+\lambda_n\nu(\theta_2),u\rangle \ge \langle \nabla\hat{R}_n(\theta_1)+\lambda_n\nu(\theta_1), u \rangle+\frac{\kappa}{2}||\theta_2-\theta_1||_2,
 \end{eqnarray}
 which is contradict with the assumption that $\theta_1$ and $\theta_2$ are two distinct stationary points of $\hat{L}_n(\theta).$
 \end{proof}
\textbf{Proof of Theorem \ref{thm3}}.
Now we are ready to prove Theorem \ref{thm3}. By Lemma \ref{lemma4} and Lemma \ref{lemma5}, as $n\ge Cs_0\log p,$ letting $\lambda_n\ge C_0M\sqrt{\frac{\log p}{n}}+\delta \frac{C_1}{\sqrt{s_0}},$ all stationary points of $L_n(\theta)$ are in $B_2^p(\theta_0,r_s)\cap \mathbb{A}\cap B^p_0(C_1s_0\log p),$ where
$r_s$ is defined in (\ref{rs}), $\mathbb{A}$ is the cone defined in Lemma \ref{lemma4}. This proves Theorem \ref{thm3}(a). Moreover, by Lemma \ref{lemma6}, Lemma \ref{lemma7}, as $n\ge C_2 s_0\log^2 p,$ $\hat{L}_n(\theta)$ cannot have two distinct stationary points in $B_2^p(\theta_0,\eta_1)\cap \mathbb{A}\cap B^p_0(C_1s_0\log p).$ Thus, as long as $\eta_1\ge r_s,$ there is only one unique stationary point of the regularized empirical risk function $\hat{L}_n(\theta),$ which is the corresponding regularized M-estimator of (\ref{form3}). This proves Theorem \ref{thm3} (b).

\textbf{Proof of Corollary \ref{cor1}:}
Huber's loss function is defined by
 \begin{eqnarray}
 \rho_{\alpha}(t)=\left\{
                    \begin{array}{ll}
\frac{1}{2}t^2, \quad\text{if $|t|\le \alpha$}\\
\alpha(|t|-\alpha/2), \quad \text{if $|t|> \alpha.$}
 \end{array}
                   \right.
  \end{eqnarray}

 the corresponding score function would be
  \begin{eqnarray}
 \psi_{\alpha}(t)= \rho_{\alpha}'(t)=\left\{
                    \begin{array}{ll}
t, \quad\text{if $|t|\le \alpha$}\\
sign(t)\alpha, \quad \text{if $|t|> \alpha.$}
 \end{array}
                   \right.
  \end{eqnarray}
Note for any $\alpha>0,$ all of $\psi(t)$, $\psi'(t)$ and $\psi''(t)$ are bounded. Specifically, we have $|\psi_{\alpha}(t)|\le \alpha,$ $|\psi'(t)|=|\psi''(t)|= 0.$
Therefore, the assumptions in Theorem \ref{thm1} and Theorem \ref{thm2} are satisfied. It is suffice to find the explicit expression of $\eta_0$ and $\eta_1$ in equation (\ref{eq4}) and (\ref{eq12}). Since $|\psi'(t)|=|\psi''(t)|= 0,$ it is easy to see $\eta_1=+\infty,$ which implies the Huber's estimator has nice computational tractability, regardless the choice of tuning parameter $\alpha$ and the percentage of outliers $\delta.$  Moreover, to find the explicit expression of $\eta_0,$ according to Assumption \ref{assume2}, we have $c_2=3, \gamma=1.$ Thus, we can calculate

\begin{eqnarray*}
h(z)&=&\int_{-\infty}^{+\infty}\psi_{\alpha}(z+\epsilon)f_0(\epsilon)d\epsilon=\int_{-\infty}^{\infty}\psi_{\alpha}(t)f_0(t-z)d t\\
&=&\int_{0}^{\alpha}t\left[f_0(t-z)-f_0(t+z)\right]d t+\alpha\int_{\alpha}^{+\infty}\left[f_0(t-z)-f_0(t+z)\right]d t\\
&\ge& \int_{0}^{\alpha}t\frac{1}{\sqrt{2\pi}\sigma}e^{-\frac{t^2+z^2}{2\sigma^2}}\left(\frac{tz}{\sigma^2}\right)dt +\alpha\int_{\alpha}^{+\infty}\frac{1}{\sqrt{2\pi}\sigma}e^{-\frac{t^2+z^2}{2\sigma^2}}\left(\frac{tz}{\sigma^2}\right)dt\\
&\ge&\frac{1}{\sqrt{2\pi}\sigma}e^{-\frac{\alpha^2+z^2}{2\sigma^2}}\int_{0}^{\alpha}t\left(\frac{tz}{\sigma^2}\right)dt +\frac{z\alpha}{\sigma^2}e^{-\frac{z^2}{2\sigma^2}}\int_{\alpha}^{+\infty}t\frac{1}{\sqrt{2\pi}\sigma}e^{-\frac{t^2}{2\sigma^2}} d t\\
&=&\frac{z\alpha^3}{3\sqrt{2\pi}\sigma^3}e^{-\frac{z^2+\alpha^2}{2\sigma^2}}+\frac{z\alpha}{\sqrt{2\pi}\sigma}e^{-\frac{z^2+\alpha^2}{2\sigma^2}}
\end{eqnarray*}

Therefore we have $H(s)=(\frac{\alpha^3}{3\sqrt{2\pi}\sigma^3}+\frac{\alpha}{\sqrt{2\pi}\sigma})e^{-\frac{s^2+\alpha^2}{2\sigma^2}}.$ By equation (\ref{eq4}) in the proof of Theorem \ref{thm1} yields
\begin{eqnarray}
\eta_0(\delta,\alpha)&=&\frac{\delta L_{\psi}}{(1-\delta)\frac{3}{4}H(\frac{8\tau r}{3}\sqrt{\frac{c_2}{\gamma}})\tau\gamma}\\
&=&\frac{\delta}{1-\delta}\frac{4\sqrt{2\pi}\sigma^3}{(\alpha^2+3\sigma^2)\tau}e^{\frac{\alpha^2+22\tau^2r^2}{2\sigma^2}},
\end{eqnarray}
which complete the proof.
\qed

\textbf{Proof of Corollary \ref{cor2}}:
 When the loss function is defined by $\rho_{\alpha}(t)=\frac{1-e^{-\alpha t^2/2}}{\alpha},$ the corresponding score function would be $\psi_{\alpha}(t)= \rho_{\alpha}'(t)=t e^{-\alpha t^2/2}.$ Moreover, we can get $\psi_{\alpha}'(t)=e^{-\alpha t^2/2}(1-\alpha t^2)$ and $\psi_{\alpha}''(t)=e^{-\alpha t^2/2}\alpha(\alpha t^2-3)$. Note for any $\alpha>0,$ all of $\psi_{\alpha}(t)$, $\psi_{\alpha}'(t)$ and $\psi_{\alpha}''(t)$ are bounded.
\begin{eqnarray}
|\psi_{\alpha}(t)|&\le& \sqrt{\frac{e}{\alpha}} \nonumber\\
|\psi_{\alpha}'(t)|&\le& \max\{1, 2e^{-1.5}\}=1\nonumber\\
|\psi_{\alpha}''(t)|&\le& \max\{e^{-(3+\sqrt{6})/2}\sqrt{(18+6\sqrt{6})\alpha}, e^{-(3-\sqrt{6})/2}\sqrt{(18-6\sqrt{6})\alpha}\}\le1.5\sqrt{\alpha}.\nonumber
\end{eqnarray}
Therefore, the Assumption \ref{assume1} is satisfied. It is suffice to find the explicit expression of $\eta_0$ and $\eta_1$ in equation (\ref{eq4}) and (\ref{eq12}). In order to have an accurate expression, we will use the individual bound of $\psi_{\alpha}(t),\psi'_{\alpha}(t),\psi''_{\alpha}(t)$ instead of the universal bound $L_{\psi}.$  Specifically, according to Assumption \ref{assume4}, $x_i$ is $\tau^2$-sub-Gaussian, $c_2=3, \gamma=1/3.$ Thus, we can calculate $h(z)=\int_{-\infty}^{+\infty}\psi_{\alpha}(z+\epsilon)f_0(\epsilon)d\epsilon=\frac{z}{(1+\alpha\sigma^2)^{3/2}}e^{-\frac{\alpha z^2}{2(1+\alpha\sigma^2)}}$
and $H(s)=\frac{1}{(1+\alpha\sigma^2)^{3/2}}e^{-\frac{\alpha s^2}{2(1+\alpha\sigma^2)}}.$ By equation (\ref{eq4}) in the proof of Theorem \ref{thm1} yields
\begin{eqnarray}
\eta_0(\delta,\alpha)&=&\frac{\delta L_{\psi}}{(1-\delta)\frac{3}{4}H(\frac{8\tau r}{3}\sqrt{\frac{c_2}{\gamma}})\tau\gamma}\\
&=&\frac{\delta}{1-\delta}\sqrt{\frac{e}{\alpha}}\frac{4(1+\alpha\sigma^2)^{3/2}}{\tau}e^{\frac{32\alpha r^2\tau^2}{3(1+\alpha\sigma^2)}}
\end{eqnarray}
Similarly, we can calculate $h'(0)=E_{f_0}\psi_{\alpha}'(\epsilon)=\frac{1}{(1+\alpha\sigma^2)^{3/2}}.$ Note $|\psi'_{\alpha}(t)|\le 1, |\psi''_{\alpha}(t)|\le 1.5\sqrt{\alpha},$ by equation (\ref{eq12}) in the proof of Theorem \ref{thm1} yields
\begin{eqnarray}
\eta_1(\delta,\alpha)&=&\frac{ (1-\delta)h'(0)\tau^2-\delta }{2\sqrt{3} \times1.5\sqrt{\alpha}\tau}\\
&=&\frac{1}{3\sqrt{3\alpha}(1+\alpha\sigma^2)^{3/2}\tau}\left[\tau^2-\delta(\tau^2+(1+\alpha\sigma^2)^{3/2})\right].
\end{eqnarray}
According to equation (\ref{eqn20}) in the proof of Theorem \ref{thm3}, we have with high probability, all stationary points of the empirical risk function $\hat{L}_n(\theta)$ in (\ref{exp2}) are inside the ball $B^p_2(\theta_0,r_s),$ where
\begin{eqnarray}
r_s&=&\eta_0+\frac{12C_{\pi}\tau \sqrt{(s_0\log p)/n}+2\tau\delta L_{\psi}}{(1-\delta)\frac{3}{4}H(\frac{8\tau r}{3}\sqrt{\frac{c_2}{\gamma}})\tau\gamma}\\
&=& (1+2\tau)\eta_0+\frac{16C_{\pi}\tau \sqrt{(s_0\log p)/n}}{(1-\delta)H(\frac{8\tau r}{3}\sqrt{\frac{c_2}{\gamma}})\tau\gamma}.
\end{eqnarray}
Therefore, as $n>>s_0\log p,$ we have $r_s\approx (1+2\tau)\eta_0,$
which completes the proof.
\qed

\bibliographystyle{apalike}
\bibliography{regression}

\end{document}